\documentclass[10pt,reqno]{amsart}
\usepackage{epsfig,amssymb,amsmath,version}
\usepackage{amssymb,version,graphicx,fancybox,mathrsfs,multirow}
\usepackage{epstopdf}
\usepackage{url,hyperref}
\usepackage{bm}
\usepackage{subfigure}

\usepackage{color,xcolor}
\usepackage{cases}
\usepackage{mathtools}
\usepackage{extarrows}
\usepackage{bbm}

\catcode`\@=11 \theoremstyle{plain}
\@addtoreset{equation}{section}   
\renewcommand{\theequation}{\arabic{section}.\arabic{equation}}
\@addtoreset{figure}{section}
\renewcommand\thefigure{\thesection.\@arabic\c@figure}
\renewcommand{\thefigure}{\arabic{section}.\arabic{figure}}
\newtheorem{thm}{\bf Theorem}

\newenvironment{theorem}{\begin{thm}} {\end{thm}}
\newtheorem{cor}{\bf Corollary}
\newtheorem{prop}{Proposition}[section]

\newtheorem{lmm}{\bf Lemma}

\DeclareMathOperator{\sgn}{sign}

\theoremstyle{remark}
\newtheorem{rem}{\bf Remark}[section]

\theoremstyle{definition}

\numberwithin{table}{section}

\def \ri {{\rm i}}

\newcommand{\bs}[1]{\boldsymbol{#1}}

\def \rd {{\rm d}}
\def \R {{\mathbb R}}

\renewcommand \wedge \times

\newcommand{\comm}[1]{\marginpar{%
		\vskip-\baselineskip 
		\raggedright\footnotesize
		\itshape\hrule\smallskip#1\par\smallskip\hrule}}

\begin{document}
	\bibliographystyle{plain}
	\graphicspath{{./figures/}}

\title[FEM on nonuniform meshes for nonlocal Laplacian]
	 {FEM on nonuniform meshes for nonlocal Laplacian: Semi-analytic Implementation in One Dimension}
	\author[H. Chen, \; C. Sheng, \; and \; L. Wang]
    {\;\; Hongbin Chen${}^1$, \;\; Changtao Sheng${}^2$  \;\; and\;\; Li-Lian Wang${}^{3}$}
	
	\thanks{${}^1$College of Computer Science and Mathematics, Central South University of Forestry and Technology, Changsha, Hunan, 410004, China. The research of the author is partially supported by the Natural Science Foundation of Hunan Province (No: 2022JJ30996). Email: hongbinchen@csuft.edu.cn (H. Chen). \\
	 \indent ${}^2$Corresponding author. School of Mathematics, Shanghai University of Finance and Economics, Shanghai, 200433, China.  The research of the author is partially supported by the National Natural Science Foundation of China (Nos. 12201385 and 12271365) and the Fundamental Research Funds for the Central Universities (No: 2021110474). Email: ctsheng@sufe.edu.cn (C. Sheng). \\
         \indent ${}^{3}$Division of Mathematical Sciences, School of Physical and Mathematical Sciences, Nanyang Technological University (NTU), 637371, Singapore. The research of the author is partially supported by Singapore MOE AcRF Tier 2 Grants:  MOE2018-T2-1-059 and MOE2017-T2-2-144. Email: lilian@ntu.edu.sg (L.-L. Wang). \\
			}
	
\begin{abstract}
In this paper, we compute stiffness matrix of the nonlocal Laplacian 
discretized by the piecewise linear finite element on nonuniform meshes, 
and implement the FEM in the Fourier transformed domain. We derive useful integral expressions of  the entries that allow us to explicitly or semi-analytically evaluate the entries for various interaction kernels. Moreover,  the  limiting cases of the nonlocal stiffness matrix when the interactional radius  $\delta\rightarrow0$ or $\delta\rightarrow\infty$ automatically lead to integer and fractional FEM stiffness matrices, respectively, and the FEM discretisation is intrinsically compatible. We conduct ample numerical experiments to study and predict some of its properties and test  on different types of nonlocal problems.  To the best of our knowledge, such a semi-analytic approach has not been explored in literature even in the one-dimensional case. 
\end{abstract}
%
	
\maketitle


\section{Introduction}
We consider the following nonlocal constrained value problem on a nonzero measure volume:
\begin{equation}\label{1DNonL}
\begin{cases}
 \mathcal{L}_{\delta} u(x) = f(x),\quad & \text{on}\ \Omega = (a, b),\\[4pt]
u(x) = 0,\quad & \text{on}\ \Omega_{\mathcal{\delta}} = [a-\delta, a] \cup [b, b+\delta],
\end{cases}
\end{equation}
where the nonlocal operator is defined as
\begin{equation}\label{NonLOper}
\mathcal{L}_{\delta} u(x) = \frac12\int_{-\delta}^\delta (u(x+s)-u(x))\rho_{\delta}(|s|) \rm{d} s,
\end{equation}
where $\delta$ is called a horizon parameter and the kernel $\rho_{\delta}=\rho_{\delta}(s)$ is a nonnegative, radial function (i.e., $\rho_{\delta}(s)= \rho_{\delta}(|s|)$ for any $s$), compactly supported in $|s| \leq \delta $, and has a bounded second moment. It is known that nonlocal models \eqref{1DNonL}  provides an alternative way to modeling many phenomena and complex systems in diverse fields, for instance, the materials science with crack nucleation and growth, fracture, and failure of composites \cite{askari2008peridynamics}, the peridynamics model for mechanics 
\cite{Chen2011, Liu2017, Tian2013}, and nonlocal heat conduction \cite{Du2012}. We also refer to the up-to-date review paper \cite{Du2012, Du2023, Elia2020} and the monograph \cite{Du2019} for more detail.
%

There has been much interest in developing numerical algorithms for nonlocal models, particularly for PDEs involving the integral fractional Laplacian, which can be viewed as a limiting case of the nonlocal operator \cite{Elia2013, Tian2016}, i.e., $\delta\rightarrow\infty.$ The recent works on PDEs involving integral fractional Laplacian operator particularly include finite difference methods \cite{Duo2018, Hao2021, Huang2014}, finite element methods \cite{Acosta2017, Ainsworth2017}, and spectral methods \cite{Mao2017, Sheng2020, Tang2020}. The interested readers may also refer to \cite{Bonito2018, Elia2020, Lischke2020} for three up-to-date review papers. It should be noted that finite element method in frequency domain space for integral fractional Laplacian is provided in \cite{Chen2021non-uniform, liu2020diagonal, Sheng2023}, where the entries of the stiffness matrix can provide exact form in one dimension, and can be simply expressed as one-dimensional integral on a finite interval in two dimension. This fact encourages us to try to extend this method for more general problems involving the nonlocal operator \eqref{NonLOper}.

Besides, extensive numerical results have been conducted on various kinds nonlocal models.  Chen and Gunzburger \cite{Chen2011} provided a systematic study of algorithmic and implementation issues using continuous and discontinuous Galerkin finite element methods for peridynamics model. Tian and Du \cite{Tian2014} introduce a finite element and quadrature-based finite difference scheme, compare the similarities and differences of the nonlocal stiffness matrices for the two methods, and study fundamental numerical analysis issues. Wang and Tian \cite{ Wang2014} develop a fast and faithful collocation method for a two-dimensional nonlocal diffusion model. Liu, Cheng and Wang \cite{Liu2017} introduced and analyzed a fast Galerkin and adaptive Galerkin finite element methods to solve a steady-state peridynamic model. 
Additionally, other numerical methods for nonlocal models have been proposed, including the finite difference method \cite{Ye2023Monotone}, the finite element methods \cite{Klar2023,Aulisa2022Efficient}, the collocation methods \cite{Leng2021,Guan2022,Lu2022a,Cao2023}, and so on.
The development of the aforementioned methods relies on constructing specific quadrature formulas to discretize the underlying nonlocal operators, which unavoidably introduces truncation errors. 
However, these truncation errors often hinder our ability to analyze the attractive structure of the nonlocal stiffness matrices and affect the numerical stability of our solver on some extreme meshes, e.g., graded and geometric meshes. Recently, Li, Liu, and Wang \cite{Li2021} developed an efficient Hermite spectral-Galerkin method for nonlocal diffusion equations in unbounded domains, where the derivation therein was built upon the convolution property of the Hermite functions.

In this paper, we provide a semi-analytic means for computing the piecewise linear FEM stiffness matrix for nonlocal problem \eqref{1DNonL}. Then, we can obtain an explicit form of the one-dimensional integral for the entries of the stiffness matrix without generating any truncation errors. This one-dimensional integral can even be evaluated exactly for most of the kernel functions, such as the fractional kernel $\rho_\delta(|s|)=C_{\delta,\alpha}/s^{1+\alpha}$, with the normalized constant $C_{\delta,\alpha}$ and the index $\alpha\in [-1,2)$. Note that such an integral representation is derived from (i) implementation of FEM in the Fourier transformed domain, and (ii) use of some analytic formulas in the handbook \cite{AbramovitzS64, Gradshteyn2007Table} to evaluation of the inner integral. 
Benefit from this explicit expression, we show the two limiting cases of the nonlocal stiffness matrix: (a) when $\delta\rightarrow0$, it reduces to the stiffness matrix associated with the classical Laplacian operator $-\Delta$, see \eqref{limting0}, and (b) when $\delta\rightarrow \infty$, it becomes the stiffness matrix corresponding to the integral fractional Laplacian operator $(-\Delta)^{\frac{\alpha}{2}}$. As a byproduct, we numerically predicted the condition number of the nonlocal stiffness matrix $\bs{{S_{\delta}}}$ (even for some extreme mesh) behaves like $\text{Cond}(\bs{S_{\delta}})\leq\min\{\left(h_{\max } / h_{\min }\right)^{\alpha-1} N^{\alpha}\delta^{\alpha-2},\,\, \left(h_{\max } / h_{\min }\right)^{2} N^{2} \}$, where $h_{\max }$ and $h_{\min}$ denote the maximum and minimum of the mesh size, and $N$ denotes the number of meshes. Our approach is suitable for wider range of nonlocal interaction kernel on the uniform and nonuniform meshes. Numerical experiments show that our scheme is asymptotically compatible scheme \cite{Tian2014}. Note that we only focus on one dimensional problems in this work to better illustrate the idea. The multidimensional generalizations are possible and will be carried out in separate works.

The rest of this paper is organized as follows. In section 2, we present a piecewise linear finite element method on nonuniform and uniform meshes for model problem \eqref{1DNonL}, and derive the explicit integral form for the entries of the nonlocal stiffness matrix. In section 3, ample numerical results for nonlocal problems with various both regular and singular kernels, including accuracy testing, limiting behaviors, eigenvalue problem and nonlocal Allen-Cahn equation, are presented to show the efficiency and accuracy. Finally, we give some concluding remarks and discussions in the last section.

\section{Finite Element Method in Fourier Domain}
\setcounter{equation}{0}
\setcounter{lmm}{0}
\setcounter{thm}{0}

In this section, we compute the  nonlocal FEM stiffness matrix. To fix the idea, we focus on the piecewise linear FEM 
with the partition:
\begin{equation}\label{partition1}
\Omega_h:=\{x_j:a=x_0<x_1<\cdots<x_{N+1}=b\},
\end{equation}
and the piecewise linear finite element basis: 
\begin{equation}\label{NonuniformP}
\phi_{j} (x) =
\begin{dcases}
 \frac{ x-x_{j-1}}{h_j},\quad & x\in [x_{j-1},x_{j}],\\
 \frac{ x_{j+1} - x}{h_{j+1}},\quad & x\in (x_{j},x_{j+1}],\\
 0, \quad & \text{elsewhere on}\;\; \Omega\cup \overline \Omega_\delta,
\end{dcases} \; h_j=x_j-x_{j-1},
\end{equation}
for  $1\leq j\leq N.$ Correspondingly, we have  the FE approximation space:
 $$\mathcal{V}_h^{0}:={\rm span}\big\{\phi_j\,:\,1\leq j\leq N\big\}.$$
 Then the FE approximation to \eqref{1DNonL} is to  find $u_h \in \mathcal{V}_h^{0}$ such that
\begin{equation}\label{VariationalF}\begin{split}
\mathcal{A}_\delta(u_h, v_h)&=\frac12 \int_{\Omega}\int_{-\delta}^{\delta}(u_h(x+s)-u_h(x))(v_h(x+s)-v_h(x))\rho_{\delta}(|s|)\,\rm{d}s\rm{d}x \\
&= (I_hf,v_h)_{\Omega}, \quad \forall \,\upsilon_h \in \mathcal{V}_h^{0},
\end{split}\end{equation}
where $I_hf$ is the standard  piecewise linear interpolation of $f$ on $\Omega_h$.
We write 
$$u_h(x)=\sum\limits_{j=1}^N u_j\phi_j(x)\in \mathcal{V}_h^{0},$$
and obtain the following  linear system of \eqref{VariationalF}:
$$ \bs{S}_{\delta}\bs u=\bs f,$$
where $\bs{S}_{\delta}$ is the nonlocal FEM stiffness matrix with entries given by 
\begin{equation}\label{Sjkform}
(\bs{S}_{\delta})_{jk} =(\bs{S}_{\delta})_{kj}= 
\frac12 \int_{\Omega}\int_{-\delta}^{\delta}(\phi_j(x+s)-\phi_j(x))(\phi_k(x+s)-\phi_k(x))\rho_{\delta}(|s|)\,\rd s\rd x,  
\end{equation}
and 
$$\bs u=(u_1, u_2, \cdots, u_N)^\top,\;\;\;\bs f=(f_1, f_2, \cdots, f_N)^\top, \;\;\; f_k= (I_hf, \phi_k)_\Omega. $$

\subsection{FEM nonlocal stiffness matrix} 
Through the implementation in the Fourier transformed domain, we 
can derive the following exact integral formulas for computing the stiffness matrix. 
\begin{thm}\label{Non-uniformP} Let 
\begin{equation}\label{djcj}
d_j^k=|x_j-x_k|,\quad  \bs{c}_j:=\Big(\frac{1}{h_j}, -\frac{1}{h_j}-\frac{1}{h_{j+1}}, \frac{1}{h_{j+1}}\Big),
\end{equation}
and define 
\begin{equation}\label{gzcj}
g(z):=g(z; s):= -\frac{1}{12}( |z+s|^3 - 2|z|^3 +|z-s|^3).
\end{equation}
Then  the stiffness matrix can be evaluated by the component-wise integral
\begin{equation}\label{stiffenty} 
\begin{split}
\bs{S}_{\delta}= &
\int_{0}^{\delta} \bs{I}(s)\, \rho_{\delta}(s)\, \rd s, 
\end{split}
\end{equation}
where the matrix-valued function $\bs{I}(s)\in {\mathbb R}^{N\times N}$ with the entries given by 
\begin{equation}\label{Djk}
\begin{split}
\bs{I}_{jk}(s)=\bs{I}_{kj}(s)=\bs{c}_jg(\bs{D}_j^k)\bs{c}_k^{\top},\quad  \bs{D}_j^k:= \begin{pmatrix}
 d_{j-1}^{k-1}  &  d_{j-1}^{k}  &  d_{j-1}^{k+1}\\[2pt]
 d_{j}^{k-1}     &  d_{j}^{k}     &  d_{j}^{k+1}\\[2pt]
 d_{j+1}^{k-1} &  d_{j+1}^{k} &  d_{j+1}^{k+1}
  \end{pmatrix}.
\end{split}
\end{equation}
Here, $g( \bs{D}_j^k)$ means the function $g$ performs component-wisely upon the matrix $\bs{D}_j^k$. 
\end{thm}
\begin{proof}Let $\tilde{\phi}_j(x)$ be the zero extension of $\phi_j(x)$ on $\Omega\cup \overline\Omega_{\delta}$ to $\mathbb{R}$. Then by \cite[Lemma 2.1]{Chen2021non-uniform}, its Fourier transform is 
\begin{equation}\label{HY2.2}
\begin{aligned}
\mathscr{F}[\tilde{\phi}_{j}(x)](\xi) &= -\frac{1}{\sqrt{2\pi}}\frac{\beta_{j}e^{-\ri x_{j-1}\xi}-(\beta_{j}+ \beta_{j+1})e^{-\ri x_{j}\xi}+ \beta_{j+1}e^{-\ri x_{j+1}\xi}}{\xi^{2}}\\
&= -\frac{1}{\sqrt{2\pi} \xi^2} \bs c_j \cdot \bs e_j(\xi),
\end{aligned}
\end{equation}
 where 
 $$\beta_j=\frac{1}{h_j},\quad 
 \bs e_j(\xi):=\big(e^{-\ri x_{j-1}\xi}, e^{-\ri x_{j}\xi}, e^{-\ri x_{j+1}\xi}\big).$$ 
Using Parseval identity and the translation property: ${\mathscr F}[u(x+s)](\xi) = e^{\ri \xi s}{\mathscr F}[u(x)](\xi)$, we derive from \eqref{NonuniformP}, \eqref{Sjkform} and \eqref{HY2.2} that
\begin{equation}\label{AuvM}
\begin{split}
\mathcal{A}_\delta(\phi_j, \phi_k)
& = \frac12\int_{\Omega}\int_{-\delta}^{\delta}(\phi_j(x+s)-\phi_j(x))(\phi_k(x+s)-\phi_k(x))\rho_{\delta}(|s|)\rd s\rd x\\
& = \frac12\int_{\R}\int_{-\delta}^{\delta}(\tilde{\phi}_j(x+s)-\tilde{\phi}_j(x))(\tilde{\phi}_k(x+s)-\tilde{\phi}_k(x))\rho_{\delta}(|s|)\rd s\rd x \\
& = \frac12\int_{-\delta}^{\delta}\rho_{\delta}(|s|)\int_{\R}{\mathscr F}[\tilde{\phi}_j(x+s)-\tilde{\phi}_j(x)](\xi)\overline{{\mathscr F}[\tilde{\phi}_k(x+s)-\tilde{\phi}_k(x)](\xi)}\rd \xi\rd s \\
& = \frac12\int_{-\delta}^{\delta}\rho_{\delta}(|s|)\int_{\R}|e^{\ri \xi s}-1|^2{\mathscr F}[\tilde{\phi}_j(x)](\xi)\overline{{\mathscr F}[\tilde{\phi}_k(x)](\xi)}\rd \xi\rd s \\
& = \int_{-\delta}^{\delta}\rho_{\delta}(|s|)\int_{\R}(1-\cos(\xi s)){\mathscr F}[\tilde{\phi}_j(x)](\xi)\overline{{\mathscr F}[\tilde{\phi}_k(x)](\xi)}\rd \xi\rd s\\
& = 2\int_{0}^{\delta}\rho_{\delta}(s)\int_{\R}(1-\cos(\xi s)){\mathscr F}[\tilde{\phi}_j(x)](\xi)\overline{{\mathscr F}[\tilde{\phi}_k(x)](\xi)}\rd \xi\rd s.
\end{split}
\end{equation}
Using \eqref{HY2.2}, we obtain from  direct calculation that
\begin{equation}\label{AuvM1}
\begin{split}
 & \int_{\R} (1-\cos(\xi s)) {\mathscr F}[\tilde{\phi}_{j}(x)](\xi)\overline{{\mathscr F}[\tilde{\phi}_{k}(x)](\xi)} \rd \xi
 \\ \quad &= \frac{1}{2\pi} \bs c_j \Big(\int_{\mathbb R} {\xi}^{-4} (1-\cos(\xi s))\bs{e}_j(\xi) \cdot \bs{e}_k(-\xi) {\rm d}\xi \Big)\bs c_k^\top
  = \frac{1}{\pi}\int_{0}^{\infty}{\xi}^{-4} f_{jk}(\xi)\,{\rm d}\xi,
\end{split}
\end{equation}
where
\begin{equation}\label{nonuni1}
 f_{jk}(\xi) = \bs c_j \bs{F}_{jk}(\xi) \bs c_k^\top, \quad
 \bs{F}_{jk}(\xi)=(1-\cos(\xi s))\cos(\bs{D}_j^k\xi).
\end{equation}
Using the trigonometric identities
\begin{equation}\label{F11}
\begin{split}
 (1-\cos(\xi s))\cos(d_{j}^k\xi) &= \cos(d_{j}^k\xi)-\frac{1}{2}\big(\cos((d_{j}^k+s)\xi)+\cos((d_{j}^k-s)\xi)\big)\\
 &= \frac{1}{2}\big(2\cos(d_{j}^k\xi) -\cos(|d_{j}^k+s|\xi) -\cos(|d_{j}^k-s|\xi)\big),
\end{split}
\end{equation}
we obtain from   \eqref{Djk} and \eqref{nonuni1} readily that
\begin{equation}\label{f3rd}
\begin{split}
f'''_{jk}(\xi) = \bs c_j \bs{F}'''_{jk}(\xi) \bs c_k^\top
\quad {\rm and} \quad   \bs{F}'''_{jk}(\xi)= \frac{1}{2}g_1(\bs{D}_j^k;\xi),
\end{split}
\end{equation}
where we denote
$$g_1(z;\xi):=  2|z|^3\sin(|z|\xi) -|z+s|^3\sin(|z+s|\xi) -|z-s|^3\sin(|z-s|\xi).$$
One verifies readily that $\partial_\xi^k g_1(z;0)=f_{jk}^{(k)}(0)=0$ for $k=0,1,2,3.$ 
Thus, we derive from integration by parts that
\begin{equation}\label{I11}
\begin{split}
\int_{0}^{\infty}{\xi}^{-4} f_{jk}(\xi)\,{\rm d}\xi
&= \frac13\int_{0}^{\infty}{\xi}^{-3}f_{jk}'(\xi)\,{\rm d}\xi= \frac16\int_{0}^{\infty}{\xi}^{-2}f_{jk}''(\xi)\,{\rm d}\xi
\\&= \frac{1}{6}\int_{0}^{\infty}{\xi}^{-1}f_{jk}'''(\xi)\,{\rm d}\xi.
\end{split}
\end{equation}
%
%
Recall the sine integral formula (see \cite[P. 427]{Gradshteyn2007Table}):
\begin{equation}\label{Table}
\int_{0}^{\infty}\frac{\sin(ax)}{x}\,\rd x = \frac{\pi}{2}\sgn(a)=\frac \pi 2 \begin{cases}
-1,\;\;  & {\rm if}\;\;a<0,\\
0,\;\; & {\rm if}\;\;a=0,\\
1,\;\;  & {\rm if}\;\;a>0.
\end{cases}
\end{equation}
%
%
Inserting \eqref{f3rd} into \eqref{I11},  we find from \eqref{Table} immediately that
\begin{equation}\label{Isformula}
\int_{0}^{\infty}{\xi}^{-4} f_{jk}(\xi)\,{\rm d}\xi
 = \frac{1}{6}\int_{0}^{\infty}{\xi}^{-1}f_{jk}'''(\xi)\,{\rm d}\xi = \frac{\pi}{2}\bs c_j g(\bs{D}_j^k)\bs c_k^\top.
\end{equation}
Then we obtain from \eqref{AuvM}-\eqref{AuvM1} and \eqref{Isformula} that
\begin{equation}\label{abasis}
(\bs{S}_{\delta} )_{jk}=\frac{2}{\pi} \int_{0}^{\delta}\rho_{\delta}(s) \Big(\int_{0}^{\infty}{\xi}^{-4} f_{jk}(\xi)\,{\rm d}\xi \Big) \rd s= \bs{c}_j \Big(\int_{0}^{\delta} g(\bs{D}_j^k)\, \rho_{\delta}(s)\, \rd s\Big)\bs{c}_k^\top. 
\end{equation}
This completes the proof.
\end{proof}


Theorem \ref{Non-uniformP} is valid for the interaction  radius $\delta<\infty.$ The bandwidth of the stiffness matrix increases as $\delta/h$ increases, which becomes full as $\delta\ge b-a.$  
It is known that  the nonlocal operator 
\eqref{NonLOper} tend to integral fractional Laplacian 
as $\delta\to\infty$ 
when the corresponding kernel function is chosen to satisfy 
 \begin{equation}\label{fracker3}
 \rho_{\delta}(s)=\rho_{\infty}(s)\mathbf{\chi}_{(0,\delta)},
\end{equation}
where $\mathbf{\chi}_{(0,\delta)}$ denotes the characteristic (indicator) function of $(0,\delta)$ and
 \begin{equation}\label{fracker}
 \rho_{\infty}(s) = \frac{C_{\alpha}}{s^{1+\alpha}} \,\; \text{with} \,\; C_{\alpha}=\frac{2^{\alpha-1}\alpha\Gamma(\frac{1+\alpha}{2})}{\sqrt{\pi}\Gamma(1-\frac{\alpha}{2})},
\end{equation}
 is the fractional kernel (see \cite{Du2023, Elia2013, Tian2016}). 

With the interchange in the order of integration in the above proof, we can derive the nonlocal stiffness matrix for $\delta\to \infty.$ 
\begin{theorem}\label{cordeltainfty} 
Let $\alpha\in(0,2)$, and choose the kernel function \eqref{fracker3},
then the entries of the stiffness matrix $\bs{S}_{\infty}$ in  \eqref{stiffenty} reduces to 
\begin{equation}\label{HY2.3}
(\bs{S}_{\infty} )_{jk}= 
\begin{dcases}
\widehat{C}_{\alpha} \bs c_j\, (\bs{D}^{k}_{j})^{3-\alpha}\,\bs c_k^\top, \quad \alpha \not= 1, \\ 
\widehat{C}_{\alpha}\bs c_j\, (\bs{D}^{k}_{j})^2\ln{\bs{D}^{k}_{j}}\,\bs c_k^\top, \quad \alpha = 1, \\
\end{dcases}
\,\; \text{with} \,\; \widehat{C}_{\alpha}=\frac{1}{2\Gamma(4-\alpha)\cos(\alpha\pi/2)}, 
\end{equation}
where the notation $(\bs D_j^k)^{3-\alpha}$ and $(\bs{D}^{k}_{j})^2\ln{\bs{D}^{k}_{j}}$ performs component-wisely upon the matrix $\bs D_j^k$, and $\bs c_j$ and $\bs D_j^k$ are defined in \eqref{gzcj} and \eqref{Djk}, respectively.
\end{theorem}
\begin{proof}
Note that when $\delta\to\infty$, we have  $\rho_{\delta}(s)=\rho_{\infty}(s)$ is independent of 
 $\delta$.
Recall the integral formula (cf. \cite[(3.12)]{Di2012}): for any $ \xi\in\mathbb{R}$, we have
\begin{equation}\label{cosintgral2}
\int_{\mathbb{R}} \frac{1-\cos (\xi s)}{|s|^{1+\alpha}}\, \rd s = C_{ \alpha}^{-1}|\xi|^{\alpha}. 
\end{equation}
We change the order of integration in the last equality of \eqref{AuvM}, 
and  derive from 
\eqref{cosintgral2} that
\begin{equation}\label{AuvMnew}
\begin{split}
(\bs{S}_{\infty} )_{jk}&
= 2C_\alpha\int_{0}^{\infty}s^{-1-\alpha}\int_{\R}(1-\cos(\xi s)){\mathscr F}[\tilde{\phi}_j(x)](\xi)\overline{{\mathscr F}[\tilde{\phi}_k(x)](\xi)}\,\rd \xi\rd s
\\& = C_\alpha\int_{\R}\Big(\int_{\R}\frac{1-\cos(\xi s)}{|s|^{1+\alpha}} \,\rd s\Big){\mathscr F}[\tilde{\phi}_j(x)](\xi)\overline{{\mathscr F}[\tilde{\phi}_k(x)](\xi)}\,\rd \xi
\\& =\int_{\R}|\xi|^{\alpha}{\mathscr F}[\tilde{\phi}_j(x)](\xi)\overline{{\mathscr F}[\tilde{\phi}_k(x)](\xi)}\,\rd \xi,
\end{split}
\end{equation}
which is identical to the FEM stiffness matrix derived in 
\cite[Theorem 1]{Chen2021non-uniform} with the formula given in \eqref{cosintgral2}.
\end{proof}

The above explicit formula allows for the exact evaluation or accurate  computation of the stiffness matrix for general interaction kernel function $\rho_\delta(s),$ and  
study the intrinsic properties of this matrix. 

\subsection{FEM nonlocal stiffness matrix with $\delta\le h$}
To study the compatibility to the usual FEM stiffness matrix, when $\delta \to 0$, we consider the special case with 
\begin{equation}\label{fracker0}
\delta\leq h=\min\limits_{1 \leq j \leq N}\{h_j\}; \quad   \rho_{\delta}(s) = \frac{C_{\delta, \alpha}}{s^{1+\alpha}}, \quad 
 C_{\delta, \alpha}:=\frac{2-\alpha}{\delta^{2-\alpha}},\;\;\; 
 \alpha\in [-1,2),
\end{equation}
where $C_{\delta, \alpha}$ is chosen so that it satisfies the second moment condition: $\int_0^\delta s^2 \rho_{\delta}(s) {\rd} s=1.$
We derive from the formulas in  Theorem \ref{Non-uniformP} straightforwardly  the following interesting identity, which leads to the usual FEM stiffness matrix when $\delta\to 0.$  
\begin{cor}\label{cordelta0}  
Under \eqref{fracker0},  we have the identity
\begin{equation}\label{limting0}
  \bs{S}_{\delta} = \bs{S}_0 -c_\alpha \delta \bs{S}_0^2,\quad  c_\alpha:=\frac{2-\alpha}{6(3-\alpha)},
  \end{equation}
where $\bs{S}_0$ is the usual FEM  stiffness matrix for $-u''(x)=f(x)$ on $(a, b)$ with $u(a)=u(b)=0,$ that is,  
$$\bs{S}_0={\rm diag}\Big(\!\!-\frac 1 {h_j}, \frac 1 {h_j}+\frac 1 {h_{j+1}},-\frac 1 {h_{j+1}}\Big).$$ 
Consequently, we have $\lim\limits_{\delta\to 0^+}\bs{S}_{\delta}=\bs{S}_0.$
\end{cor}
\begin{proof}
As $s\le \delta\le h,$ we can simplify $g(z;s)$ in \eqref{gzcj} and evaluate \eqref{Djk} directly to derive
\begin{equation*}\label{Ijks1}
\begin{split}
\bs{I}_{jk}(s) &=\bs c_j g(\bs{D}_j^k)\bs c_k^\top \\
 & =\begin{dcases}
-2\Big(\frac{1}{h_j^2} +\Big(\frac{1}{h_j}+\frac{1}{h_{j+1}}\Big)^2 +\frac{1}{h_{j+1}^2}\Big)s^3 +12\Big(\frac{1}{h_j}+\frac{1}{h_{j+1}}\Big)s^2,\;\;  &  j=k,\\[1pt]
\frac{2}{h_{j+1}}\Big(\frac{1}{h_j}+\frac{2}{h_{j+1}}+\frac{1}{h_{j+2}}\Big)s^3 -\frac{12}{h_{j+1}}s^2,\;\;  & j=k -1,\\[1pt]
 \frac{2}{h_{j}}\Big(\frac{1}{h_{j-1}}+\frac{2}{h_j}+\frac{1}{h_{j+1}}\Big)s^3 -\frac{12}{h_j}s^2,\;\;  &  j=k +1,\\[1pt]
-\frac{2}{h_{j+1}h_{j+2}}s^3,\;\;  & j=k -2,\\[1pt]
-\frac{2}{h_{j-1}h_{j}}s^3,\;\;    &  j=k +2,\\[1pt]
0,\;\;  & {\rm otherwise}.
\end{dcases}
\end{split}
\end{equation*}
Thus for given $\rho_\delta(s)$ in \eqref{fracker0}, we find
\begin{equation}\label{Sjk2}
\begin{split}
(\bs{S}_{\delta} )_{jk}& = \int_{0}^{\delta}\bs{I}_{jk}(s)\, \rho_{\delta}(s)\, \rd s \\
 & =  \begin{dcases}
-c_\alpha \Big(\frac{1}{h_j^2} +\Big(\frac{1}{h_j}+\frac{1}{h_{j+1}}\Big)^2 +\frac{1}{h_{j+1}^2}\Big) \delta+\frac{1}{h_j}+\frac{1}{h_{j+1}},\;\;  &  j=k,\\[1pt]
c_\alpha\Big(\frac{1}{h_jh_{j+1}}+\frac{2}{h^2_{j+1}}+\frac{1}{h_{j+1}h_{j+2}}\Big) \delta-\frac{1}{h_{j+1}},\;\;  & j=k -1,\\[1pt]
c_\alpha\Big(\frac{1}{h_{j-1}h_j}+\frac{2}{h^2_j}+\frac{1}{h_jh_{j+1}}\Big) \delta -\frac{1}{h_j},\;\;   & j=k +1,\\[1pt]
 -c_\alpha\frac{1}{h_{j+1}h_{j+2}} \delta,\;\;  & j=k -2,\\[1pt]
 -c_\alpha\frac{1}{h_{j-1}h_{j}} \delta,\;\;  & j=k +2,\\[1pt]
 0,\;\;  & {\rm otherwise},
\end{dcases}\\
&=(\bs S_0)_{jk}-c_\alpha \delta (\bs S_0^2)_{jk},
\end{split}
\end{equation}
by direct calculaton, we can obtain the identity \eqref{limting0}. 
\end{proof}

\begin{rem}\label{Rmk:delta} {\em The compatibility  of the discretisation is essential for the numerical solution of nonlocal models. With the evaluation of the stiffness matrix by the analytic formulas, the FE scheme is automatically compatible \cite{Tian2014, Tian2016}.} 
\end{rem}

\subsection{Nonlocal FEM stiffness matrix on uniform meshes}
As a special case of Theorem \ref{Non-uniformP}, the nonlocal stiffness matrix of FEM on uniform meshes reduces to a symmetric Toeplitz matrix. 
\begin{thm}\label{uniformgrids}  
Given a uniform  partition   $\{x_j=a+jh\}_{j=0}^{N+1}$ with $h=(b-a)/(N+1)$, the nonlocal stiffness matrix $\bs{S}_{\delta}$ 
is a symmetric Toeplitz matrix of the form
\begin{equation}\label{StiffMatrix}
\bs{S}_{\delta}=
\begin{bmatrix}
   t_{0} &\hspace{-4pt} t_{1} & t_2 & \cdots & t_{N-3}& t_{N-2} & t_{N-1} \\[1pt]
   t_{1} &\hspace{-4pt} t_{0} & t_{1} &\hspace{-4pt} \ddots & \cdots & t_{N-3}  & t_{N-2}\\[-1pt]
   t_{2} &\hspace{-4pt} t_{1} & t_{0} & \ddots & \hspace{-4pt}\ddots &\vdots   & t_{N-3} \\[0pt]
   \vdots& \hspace{-4pt}\ddots& \hspace{-4pt}\ddots & \ddots   & \hspace{-4pt}\ddots &\hspace{-4pt} \ddots & \vdots \\[2pt]
   t_{N-3}   &  \vdots   &\hspace{-4pt} \ddots   &\hspace{-4pt} \ddots  &\hspace{8pt} t_0 &t_1  & t_2   \\[-2pt]
   t_{N-2}  & t_{N-3} & \cdots   &\hspace{-4pt}\ddots &\hspace{8pt}  t_{1}  & t_{0} & t_{1} \\[3pt]
   t_{N-1}  & t_{N-2} & t_{N-3}  & \cdots &\hspace{8pt} t_2  & t_{1} & t_{0}
  \end{bmatrix},
\end{equation}
i.e., $\{(\bs{S}_{\delta})_{jk}=t_{|j-k|}\}_{j,k=1}^N.$ Here, the
generating vector 
$\bs t=(t_0, t_1, \cdots, t_{N-1})$ can be evaluated by
\begin{equation}\label{gpa00}
 t_p= \begin{dcases}  h^2\int_{0}^{\frac{\delta}{h}}{\mathbb I}_p(\tau) \, \rho_{\delta}(\tau h)\, \rd \tau,\quad & 0\le p< \delta/h+2,\\
 0,\quad & p\ge \delta/h+2,
\end{dcases}\end{equation}
where $p=|j-k|$ and
\begin{equation}\label{stiffentyU}
\begin{split}
\mathbb{I}_p(\tau) := - \frac{1}{12} \sum\limits_{i=-2}^2 \eta_i \big\{|p+i+\tau|^3 -2|p+i|^3 +|p+i-\tau|^3\big\},
\end{split}
\end{equation}
with the constants $\eta_0=6$, $ \eta_{\pm 1}=-4$ and $\eta_{\pm 2}=1.$ 
\end{thm}
\begin{proof} 
 In this case, the entries in \eqref{Djk} can be simplified into 
\begin{equation*}
\begin{split}
{\bs I}_{jk}(s) &={\bs I}_{kj}(s)=\bs c_j g(\bs{D}_j^k)\bs c_k^\top \\
& = \frac{1}{h^2}\big(1, -2, 1\big) \begin{pmatrix}
 g(|j-k|h)     &  g(|j-k-1|h)  &  g(|j-k-2|h) \\[2pt]
 g(|j-k+1|h)   &  g(|j-k|h)    &  g(|j-k-1|h) \\[2pt]
 g(|j-k+2|h)   &  g(|j-k+1|h)  &  g(|j-k|h)
  \end{pmatrix} \big(1, -2, 1\big)^\top \\
& = \frac{1}{h^2}\sum\limits_{i=-2}^2 \eta_i g(|j-k+i|h) 
 = \frac{1}{h^2}\sum\limits_{i=-2}^2 \eta_i g((j-k+i)h)\\
 & = \frac{1}{h^2}\sum\limits_{i=-2}^2 \eta_i g((|j-k|+i)h),
\end{split}
\end{equation*}
where we observed from the definition \eqref{gzcj} that $g$  is an even function in $z$ for fixed $s\ge 0,$ and also noted that the symmetric properties of  coefficients $\{\eta_i\}$ and summation in $i.$ 
Thus we can denote 
\begin{equation*}
\begin{split}
\bs{I}_{jk}(s) &=  \frac{1}{h^2}\sum\limits_{i=-2}^2 \eta_i g((p+i)h):={\mathbb I}_p(s)={\mathbb I}_{|j-k|}(s),
\end{split}
\end{equation*}
which implies the entries on each diagonal are the same, i.e., $\bs{I}(s)$ is a symmetric Toeplitz matrix, so is $\bs S_\delta.$ 
The expression \eqref{gpa00}-\eqref{stiffentyU} is a direct consequence of \eqref{gzcj}-\eqref{stiffenty}.

If $p\ge \delta/h+2,$ we find from \eqref{stiffentyU} that for $\tau\in [0,\delta/h],$
\begin{equation*}
\begin{split}
{\mathbb I}_p(\tau) & = - \frac{1}{12} \sum\limits_{i=-2}^2 \eta_i \big\{|p+i+\tau|^3 -2|p+i|^3 +|p+i-\tau|^3\big\}\\
&= - \frac{1}{12} \sum\limits_{i=-2}^2 \eta_i \big\{(p+i+\tau)^3 -2(p+i)^3 +(p+i-\tau)^3\big\}=0.
\end{split}
\end{equation*}
This completes the proof.
\end{proof}

\begin{rem}\label{Iptaucomp} \emph{For any $\tau\ge 0$ and fixed integer $p\ge 0,$  ${\mathbb I}_p(\tau)$ in
 \eqref{stiffentyU} has the following piecewise representation:
\begin{itemize}
\item[(i)] If $p=0,1,$ then we have
\begin{equation}\label{Ip=0}
{\mathbb I}_{0}(\tau) =
  \begin{cases}
  12  (2-\tau)\tau^2,\quad &   \tau\in [0,1),\\[1pt]
16+4(\tau-2)^3, \; &    \tau\in [1,2),\\[1pt]
16, \quad & \tau\ge 2,
\end{cases}
\end{equation}
and
\begin{equation}\label{Ip=1}
 {\mathbb I}_{1}(\tau)  =
  \begin{cases}
  -4(3-2\tau)\tau^2,\quad &  \tau\in [0,1),\\[1pt]
  14-42\tau+30\tau^2-6\tau^3, \;\; &  \tau\in [1,2),\\[1pt]
4+2(\tau-3)^3, \quad &    \tau\in [2,3),\\[1pt]
  4, \quad &   \tau\ge 3.
\end{cases}
\end{equation}
\item[(ii)] If $p \geq 2,$ then we have
\begin{equation}\label{Ipge2}
 {\mathbb I}_{p}(\tau) =
  \begin{cases}
  2(p-2-\tau)^3,\quad &   \tau\in [p-2,\, p-1),\\[1pt]
  2(p-2-\tau)^3-8(p-1-\tau)^3,\;\; &  \tau\in [p-1,\,p),\\[1pt]
  8(p+1-\tau)^3-2(p+2-\tau)^3,\; &  \tau\in [p,\,p+1),\\[1pt]
-2(p+2-\tau)^3,\quad &  \tau\in [p+1,\,p+2),\\[1pt]
0,\quad & {\rm otherwise}.
\end{cases}
\end{equation}
\end{itemize}}
\end{rem}

\section{Numerical Studies and Discussions}
In this section, we first conduct a numerical study of the conditioning of the stiffness matrix on three sets of typical non-uniform grids used in practice to resolve singular solutions.  We then numerically solve several types of nonlocal problems by using the FEM nonlocal matrix to show the efficiency and accuracy of the proposed scheme. 

 \vspace{-4pt}
\subsection{Conditioning of the stiffness matrix on nonuniform meshes}  An issue of interest is the condition number estimation for the FEM nonlocal stiffness matrix. It is a well-studied topic in the uniform meshes case, but much less known in this nonuniform meshes setting. In general, the mesh geometry affects not only the approximation error of the finite element solution but also the spectral properties of the corresponding stiffness matrix.

Extensive research has been conducted on the condition number of the nonlocal stiffness matrix on uniform grids (cf. \cite{Akso2014, Zhou2010}), which the entries of the nonlocal stiffness matrix can be computed exactly based on the definition of hypersingular integrals form therein. With the aid of the analytical expressions, the following condition number bounds for uniform mesh were reported in \cite[p.1772]{Zhou2010}:
\begin{equation}\label{cond}
\text{Cond}(\bs{S_{\delta}})\leq \min\{N^{\alpha}\delta^{\alpha-2},\,\, N^{2} \}.
\end{equation}

However, the study on the condition number of the nonlocal stiffness matrix on the non-uniform mesh is still open. Inspired by the result of the condition number for FEM stiffness on nonuniform mesh in \cite[(26)]{Fri1972}, we predict the condition number of FEM stiffness matrix \eqref{stiffenty} on nonuniform meshes associated with fractional kernel \eqref{fracker0} is 
\begin{equation}\label{cond1}
\text{\rm Cond}(\bs{S_{\delta}})\leq \min\{\left(h_{\max } / h_{\min }\right)^{\alpha-1} N^{\alpha}\delta^{\alpha-2},\,\, \left(h_{\max } / h_{\min }\right)^{2} N^{2} \},
\end{equation}
where $h_{\max}=\max_{1\leq j\leq N}h_j$ and $h_{\min}=\min_{1\leq j\leq N}h_j$. It is clear that the above prediction can reduce to the result with uniform mesh \eqref{cond} when $h_{\max } =h_{\min }$. 

We will explore this numerically on graded, geometric and Shishkin meshes in order to demonstrate the prediction or conjecture. In the following numerical experiment, we will choose the mesh division as follows:
\begin{itemize}
\item[(i)] {Graded meshes}:\; A graded meshes on the interval $(a, b)$ is defined as follows 
\begin{equation}\label{GradedM1}
x_j =
\begin{cases}
 a + \dfrac{b-a}{2}\Big(\dfrac{2j}{N}\Big)^{\gamma},  &  j = 0, 1, \cdots, N/2-1,\\[6pt]
 b - \dfrac{b-a}{2}\Big(2-\dfrac{2j}{N}\Big)^{\gamma}, \;\;  &  j = N/2, N/2+1, \cdots, N.
\end{cases}
\end{equation}

 From \eqref{cond1}, the condition number of FEM stiffness matrix on the above graded meshes is
 \begin{equation}
  \text{Cond}(\bs{S_{\delta}})\sim N^{\gamma(\alpha-1)+1}.
\end{equation}

\item[(ii)] {Geometric meshes}: A geometric meshes on the interval $(a, b)$ is defined according to
\begin{equation}\label{GeometricM3}
x_j =
\begin{cases}
 a + q^{N-j} \dfrac{b-a}{2},  & \quad j=1, 2, \cdots, N-1,\\[4pt]
 b - q^{j-N} \dfrac{b-a}{2},  & \quad j=N, N+1, \cdots, 2N-1,
\end{cases}
\end{equation}
 where $q$ ($0 < q < 1$, $q$ is independent of $j$) is subdivision ratio.
 
 Using \eqref{cond1}, the condition number of FEM stiffness matrix on the above geometric meshes is
 \begin{equation}
  \text{Cond}(\bs{S_{\delta}})\sim q^{1-N}N^{\alpha}.
\end{equation}

\item[(iii)] { Shishkin meshes}:
A Shishkin meshes is piecewise uniform mesh over $[a, b]$  with $2M+N$ elements. For $\eta \in(0,\frac{1}{2})$, we refer to the mesh size over  $[0, \eta]$ and $[1-\eta,\, 1]$ as $h = \frac{\eta}{M}(b-a)$ and to the mesh size over $[\eta,\, 1-\eta]$ as $H = \frac{1-2\eta}{N}(b-a)$.
\end{itemize}

 According to \eqref{cond1}, the condition number of FEM stiffness matrix on the above Shishkin meshes is
 \begin{equation}
  \text{Cond}(\bs{S_{\delta}})\sim \Big(\frac{M(1-2\eta)}{N\eta}\Big)^{\alpha-1}N^{\alpha}.
\end{equation}

In Figure\,\ref{cond1-1} (a)-(c), we plot the condition number of the stiffness matrix  against various $N$, fixed horizon $\delta=0.3$ and different $\alpha$. We observe a good agreement between the numerical results and \eqref{cond1}, which indicates a clear dependence of the condition number on three parameters: the mesh ratio $ h_{\max}/h_{\min}$, the size of the nonlocality horizon parameter $\delta$, and index of the fractional kernel function $\alpha$.

\medskip

 \begin{figure}
 \begin{minipage}[t]{0.5\linewidth}
 \centering
 \includegraphics[width=0.96\textwidth]{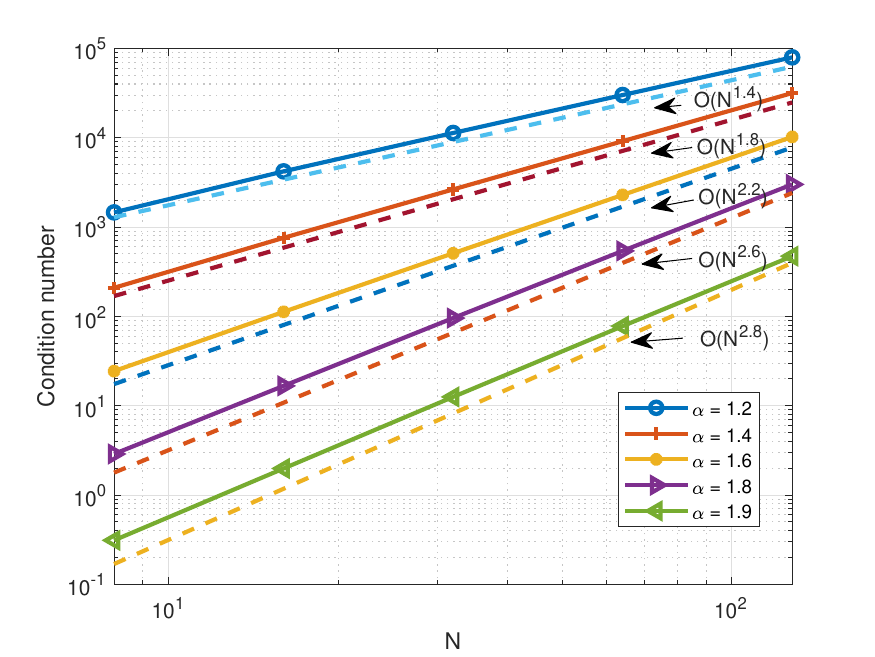}
 \centerline{(a)\,\,$\text{Cond}(\bs{S_{\delta}})\sim N^{\gamma(\alpha-1)+1}$.}
 \end{minipage}%
 \begin{minipage}[t]{0.5\linewidth}
  \centering
  \includegraphics[width=0.9\textwidth]{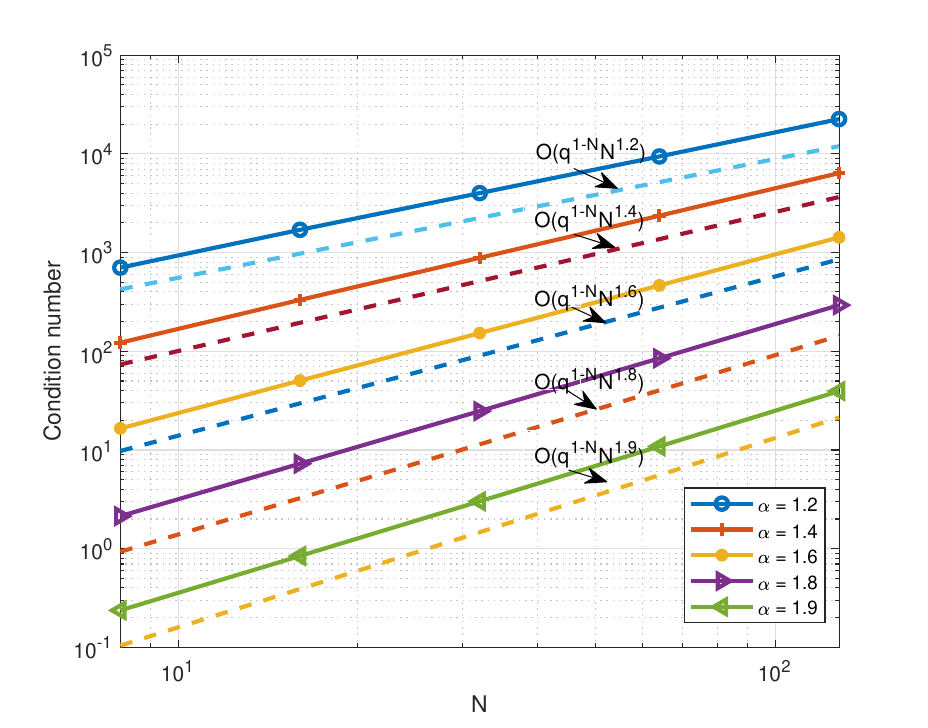}
  \centerline{(b)\,\,$\text{Cond}(\bs{S_{\delta}})\sim q^{1-N}N^{\alpha}$.}
 \end{minipage}

  \begin{minipage}[t]{0.5\linewidth}
 \centering
 \includegraphics[width=0.96\textwidth]{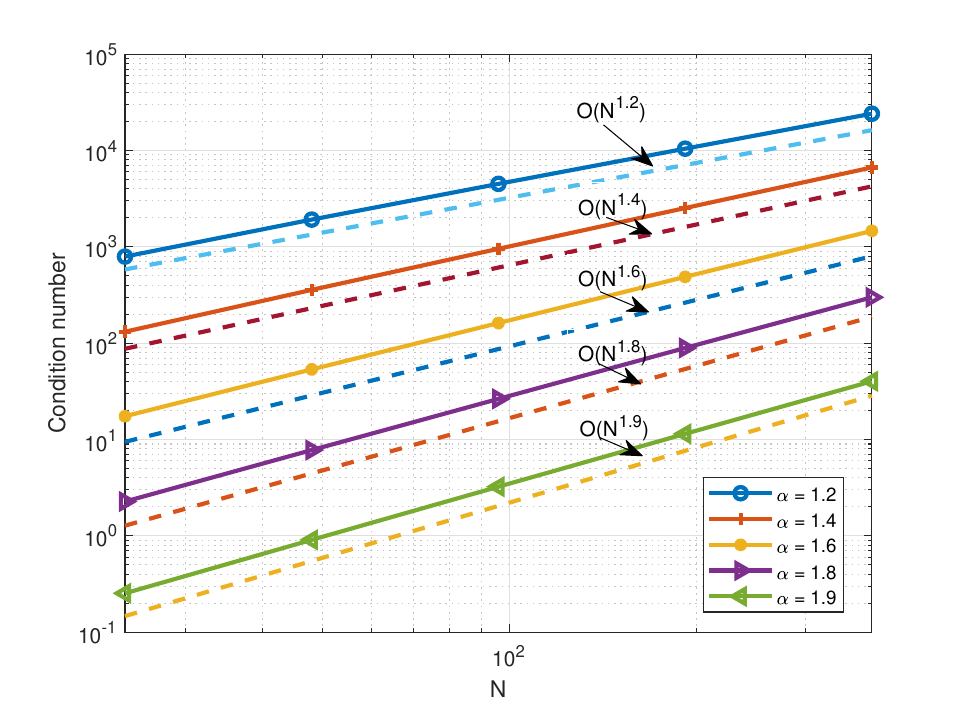}
  \centerline{(c)\,\,$\text{Cond}(\bs{S_{\delta}})\sim (\frac{2(1-2\eta)}{\eta})^{\alpha-1}N^{\alpha}$.}
 \end{minipage}%
	\begin{minipage}[t]{0.5\linewidth}
		\centering
		\includegraphics[width=0.96\textwidth]{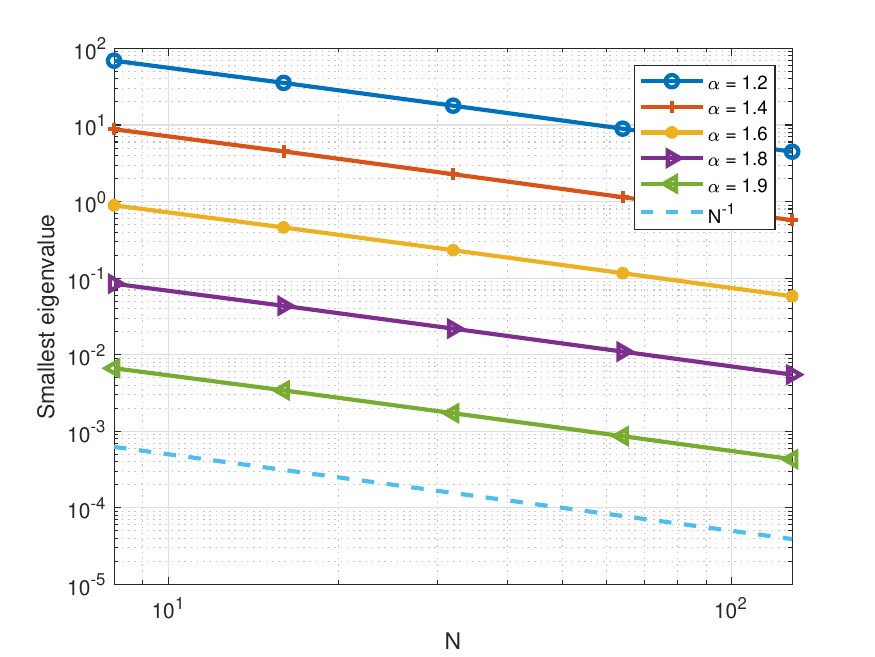}
  \centerline{(d)\,\,$\lambda_{\min} \sim N^{-1}$.}
	\end{minipage}
  
 \caption{Conditioning and the smallest eigenvalue of the stiffness matrix $\bs{S}_{\delta}$ for $\rho_\delta(s)= \frac{2-\alpha}{\delta^{2-\alpha}} s^{-(1+\alpha)}$ on different meshes with $\delta=0.3$. Condition number: (a) Graded meshes with $\gamma=2$, (b) Geometric meshes with $q = 0.999$, (c) Shishkin meshes with $M = 2N$ and $\eta = 0.2$, (d) The smallest eigenvalue on graded meshes with $\gamma=2$.} \label{cond1-1}
\end{figure}

Meanwhile, we plot the first smallest eigenvalue of the stiffness matrix against various $N$ and different $\alpha$ in Figure\,\ref{cond1-1} (d), where we observed that the smallest eigenvalue of $\bs{S}_{\delta}$ satisfy
$$
\lambda_{\min }(\bs{S}_{\delta})=c h_{\max }=c N^{-1}, \quad \alpha \in[1, 2].
$$
This result is consistent with the classical usual Laplacian (i.e., $\delta=0$) and the integral fractional Laplacian (i.e., $\delta  \to \infty$)  (cf. \cite{Chen2021non-uniform}).

 \subsection{Nonlocal BVPs with various interaction kernels} 

 In this subsection,  we focus on numerical approximation the following nonlocal problems
\begin{equation}\label{problem32}
\begin{cases}
    {\mathcal L}_\delta u(x)=\lambda u(x)+f(x),\;\;\;&\text{on}\;\Omega,\\
    u(x)=0,&\text{on}\;\Omega_\delta,
\end{cases}
\end{equation}
 which can be considered as nonlocal BVP if $\lambda=0$; or the nonlocal Helmholtz problem if $\lambda<0$ is given; or eigenvalue problem if $\lambda$ is unknown and $f(x)=0$.
 

\subsubsection{\bf Nonlocal BVP with $\lambda=0$}
In this subsection, we will consider the accuracy test and limiting behaviours for nonlocal BVPs \eqref{problem32} with $\lambda=0$.

\medskip
\noindent {\bf Example 1.} {\bf (Accuracy test of smooth solutions).} We first test the accuracy of the proposed method for the nonlocal BVP \eqref{problem32} with $\lambda=0$ and involving a fractional kernel function \eqref{fracker0} with different values of the parameters $\alpha$.

In the simulation, we use the FEM on uniform meshes, and take the exact solution $u(x)=x^2(1-x)^2$ and $\Omega=(0,1)$  (cf. \cite{Chen2011, Tian2013}). In Figure \ref{fig2}, we plot the numerical errors against various $h$ with different parameters $\alpha=0.1,0.3,0.5,0.7,0.9$. We observe from Figure \ref{fig2} that the convergence order of the piecewise linear FEM on uniform meshes is $O(h^{2})$ for fixed horizon $h \leq \delta \leq \frac{b-a}{2}$, while for fixed horizon $0 \leq \delta \leq h$, the convergence rate is $O(h)$.

  \begin{figure}[!h]
	\begin{minipage}[t]{0.5\linewidth}
  \centering
  \includegraphics[width=0.96\textwidth]{ 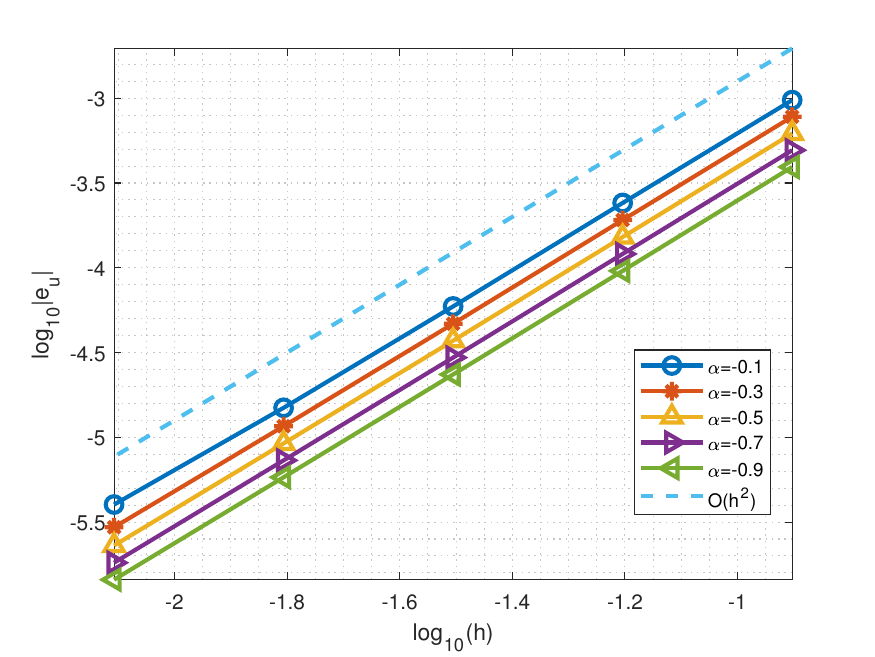}
  \centerline{(a)\,\, $h \leq \delta \leq \frac{b-a}{2}$.}
	\end{minipage}%
	\begin{minipage}[t]{0.5\linewidth}
   \centering
   \includegraphics[width=0.96\textwidth]{ 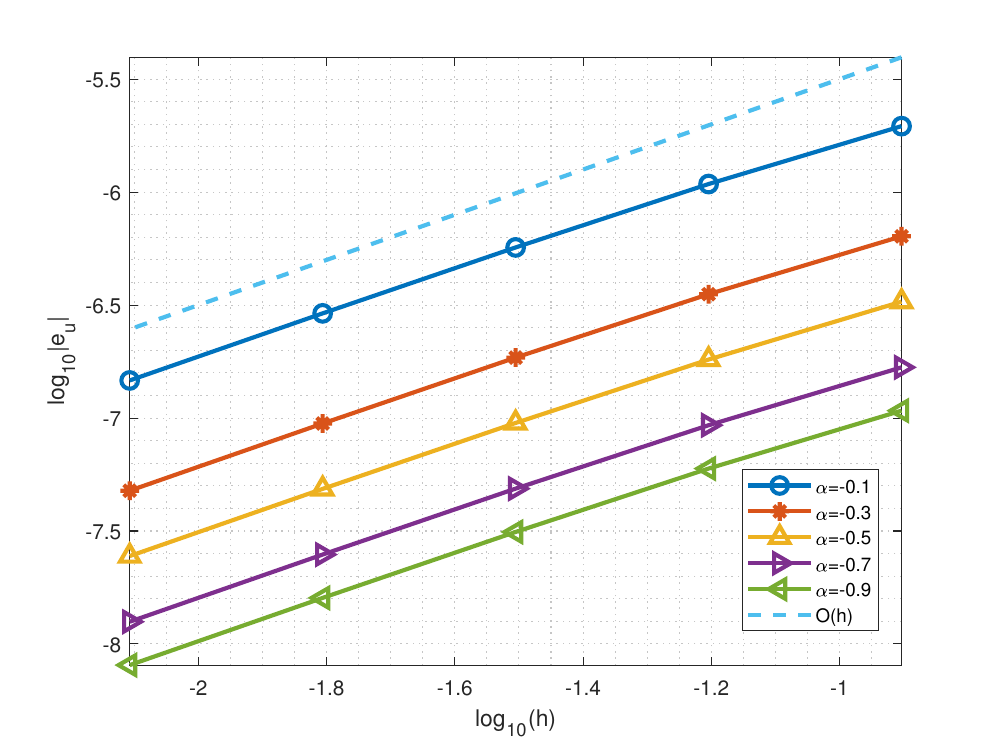}
   \centerline{(b)\,\, $0 < \delta < h $.}
	\end{minipage}
\caption{The numerical errors for $\rho_\delta(s)= \frac{2-\alpha}{\delta^{2-\alpha}} s^{-(1+\alpha)}$ with different $\alpha \in (-1, 0)$.}\label{fig2}
\end{figure}

\medskip
\noindent{\bf Example 2.}  {\bf (Accuracy test of discontinuous solution).}
Next, we consider the FEM on uniform and graded meshes for the nonlocal BVP \eqref{problem32} with $\lambda=0$ and the discontinuous exact solution on $\Omega=(0,1)$ given by (cf. \cite{Chen2011, Tian2013})  
\begin{equation}
u(x) =
  \begin{cases}
  2x^2,\quad &   x<0.5,\\[1pt]
 (1-x)^2, \quad & \text {otherwise}.
\end{cases}
\end{equation}
In this computation, we take the constant box potential $\rho_{\delta}(s) = \frac{3}{\delta^3}$, and choose the following graded meshes
\begin{equation}\label{GradedM}
x_j =
\begin{cases}
  \dfrac{b+a}{2} -  \dfrac{b-a}{2}\Big(1-\dfrac{2j}{N}\Big)^{\gamma},  &  j = 0, 1, \cdots, \dfrac{N}{2}-1,\\[6pt]
 \dfrac{b+a}{2} + \dfrac{b-a}{2}\Big(\dfrac{2j}{N}-1\Big)^{\gamma}, \;\;  &  j = \dfrac{N}{2}, \dfrac{N}{2}+1, \cdots, N,
\end{cases}
\end{equation}
where $\gamma > 1$ is a suitable grading exponent. Different from \eqref{GradedM1}, the above graded meshes clustering near the discontinuous points $x=0.5$, which is beneficial for capturing the discontinuous nature of the solution. 

In Figure \ref{fig3}(a)-(b), we plot the numerical errors of the FEM on uniform meshes with various $\delta$, which shows that the convergence order is $O(h^{0.5})$ in $L^{\infty}$-norm and $O(h)$ in $L^{2}$-norm. As a comparison, we plot numerical errors of the FEM on graded meshes with various $\delta$ in Figure \ref{fig3}(c)-(d). We observe that, for fixed grading exponent $1 \leq \gamma \leq 2$, the $L^2$-norm errors are roughly of $O(h^{\frac{\gamma}{2}})$. For fixed grading exponent $ \gamma \geq 2$,  i.e., the $L^2$-norm errors are roughly of $O(h^{\min(2, \gamma-1)})$. The convergence order of FEM on graded meshes is significantly higher than that of FEM on uniform meshes when there are discontinuities in the solution. 
 \begin{figure}[ht]
	\begin{minipage}[t]{0.5\linewidth}
		\centering
		\includegraphics[width=0.9\textwidth]{ 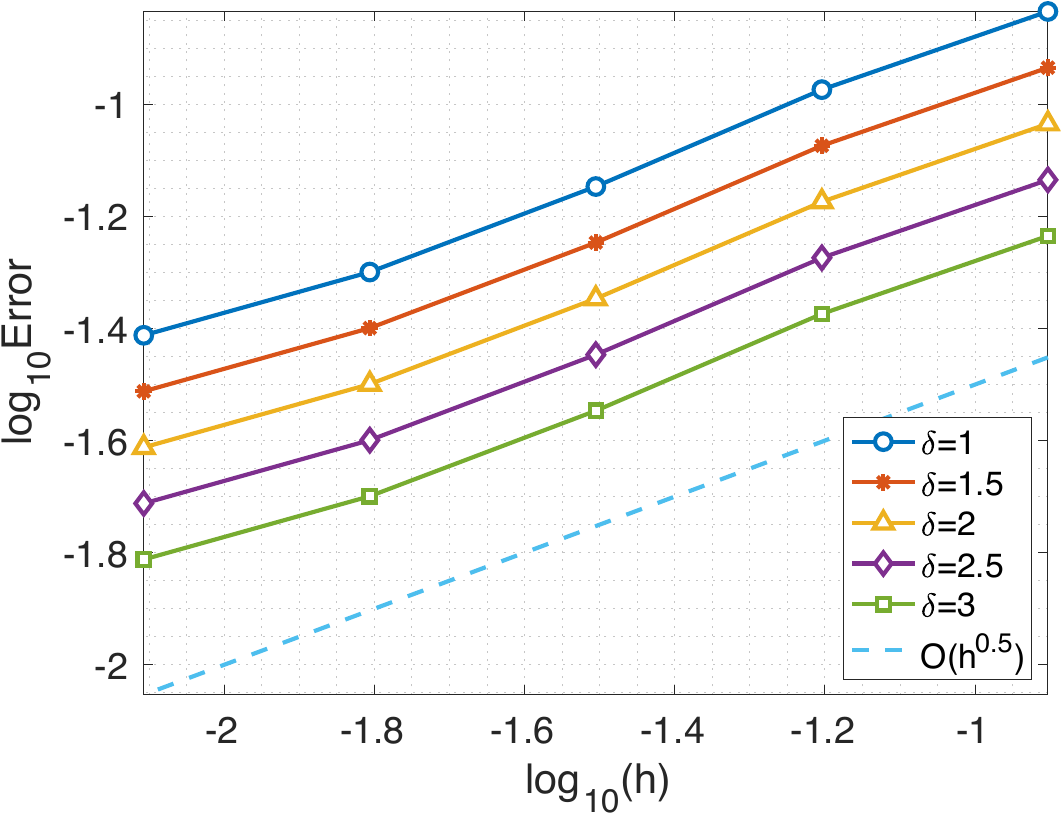}
   \centerline{(a)\,\,$L^{\infty}$-norm errors.}
	\end{minipage}%
	\begin{minipage}[t]{0.5\linewidth}
		\centering
		\includegraphics[width=0.9\textwidth]{ 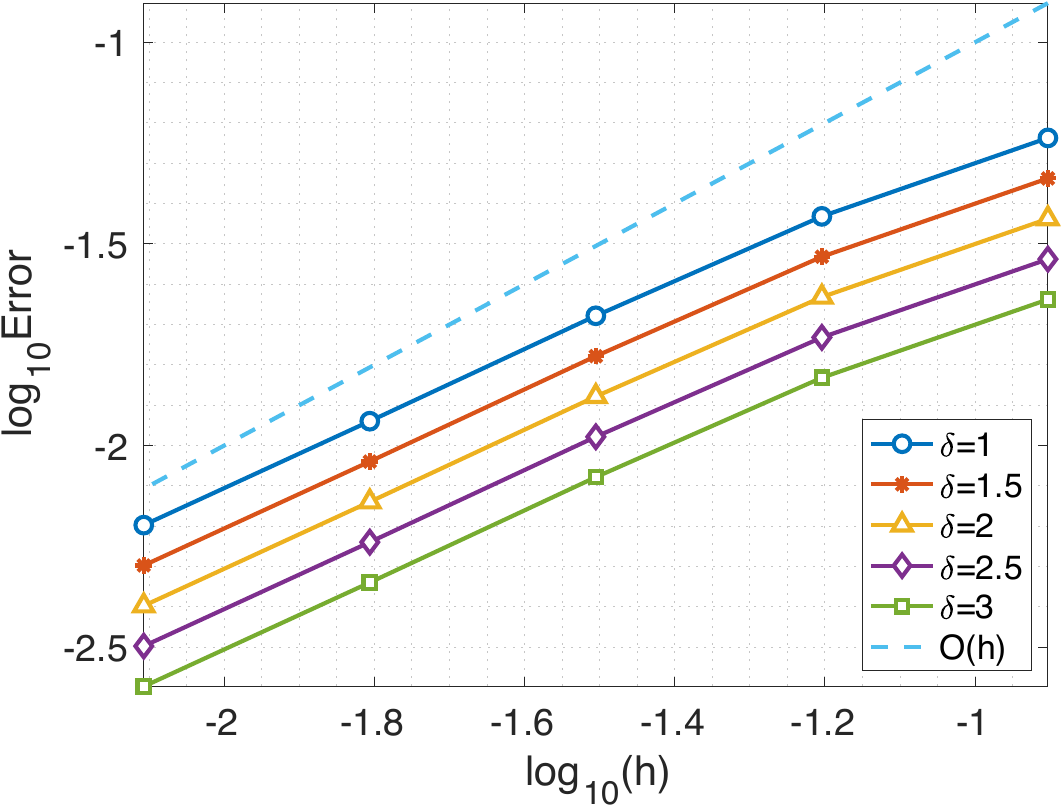}
    \centerline{(b)\,\,$L^2$-norm errors.}
	\end{minipage}
	\begin{minipage}[t]{0.5\linewidth}
	\centering
	\includegraphics[width=0.96\textwidth]{ 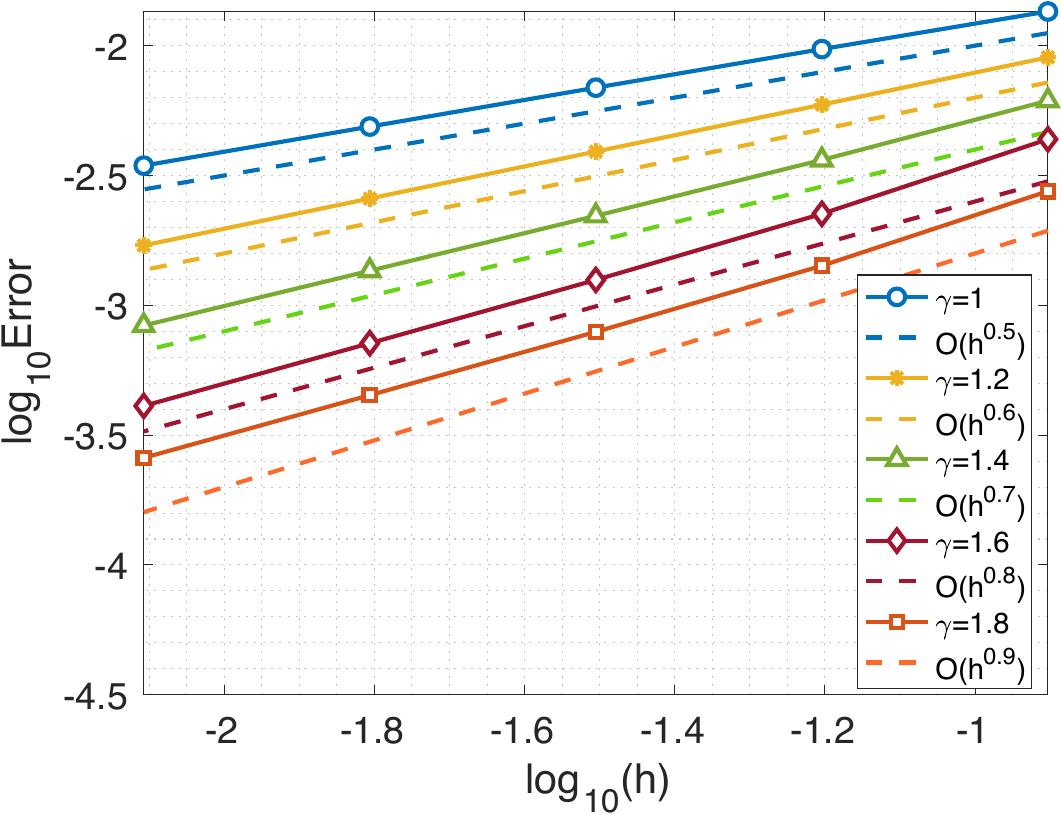}
    \centerline{(c)\,\,$1\leq \gamma <2 $.}
	\end{minipage}%
	\begin{minipage}[t]{0.5\linewidth}
	\centering
	\includegraphics[width=0.96\textwidth]{ 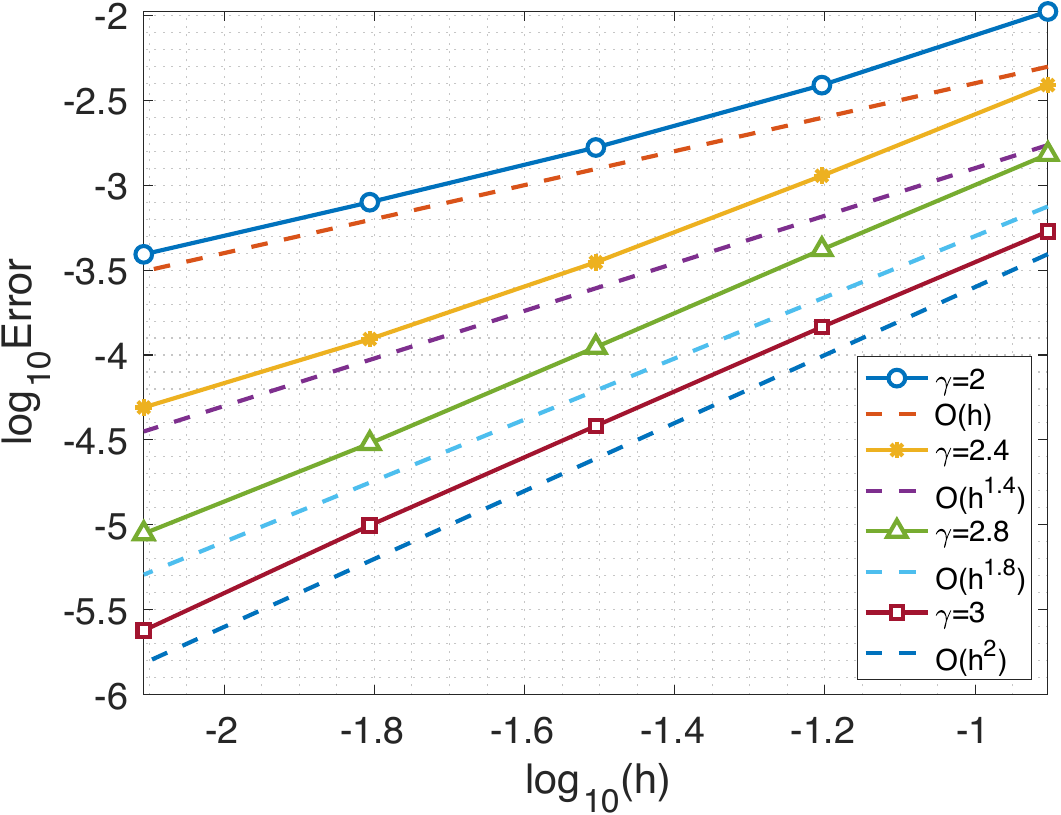}
    \centerline{(d)\,\,$ \gamma \geq 2 $.}
	\end{minipage}
	\caption{The numerical errors of discontinuous solution. (a)-(b): FEM on uniform meshes, (c)-(d): FEM on graded meshes in $L^2$-norm.} \label{fig3}
\end{figure}

\begin{figure}
	\begin{minipage}[t]{0.52\linewidth}
	\centering
	\includegraphics[width=1.06\textwidth]{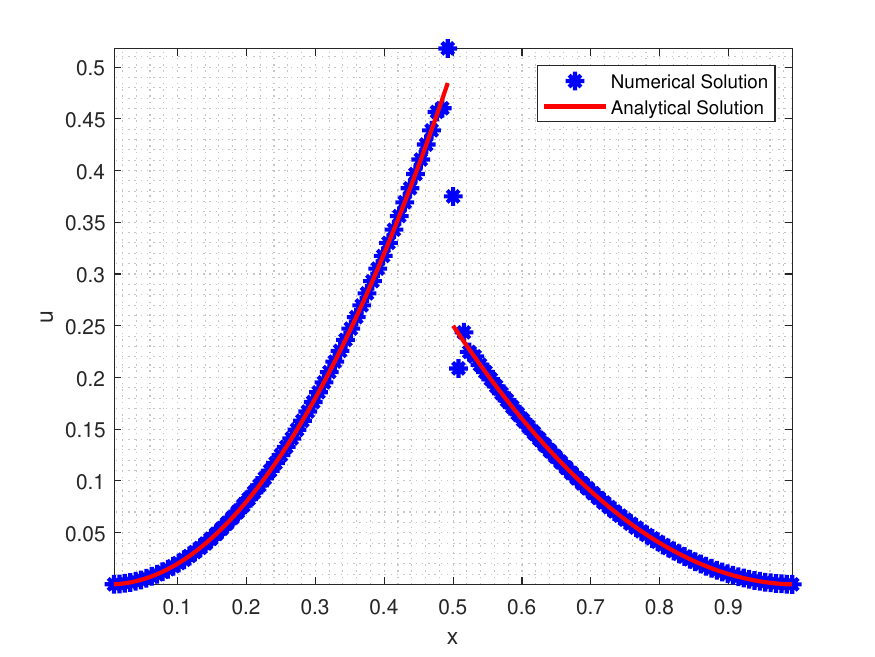}
    \centerline{(a)\,\,uniform meshes.}
	\end{minipage}%
	\begin{minipage}[t]{0.52\linewidth}
	\centering
	\includegraphics[width=1.06\textwidth]{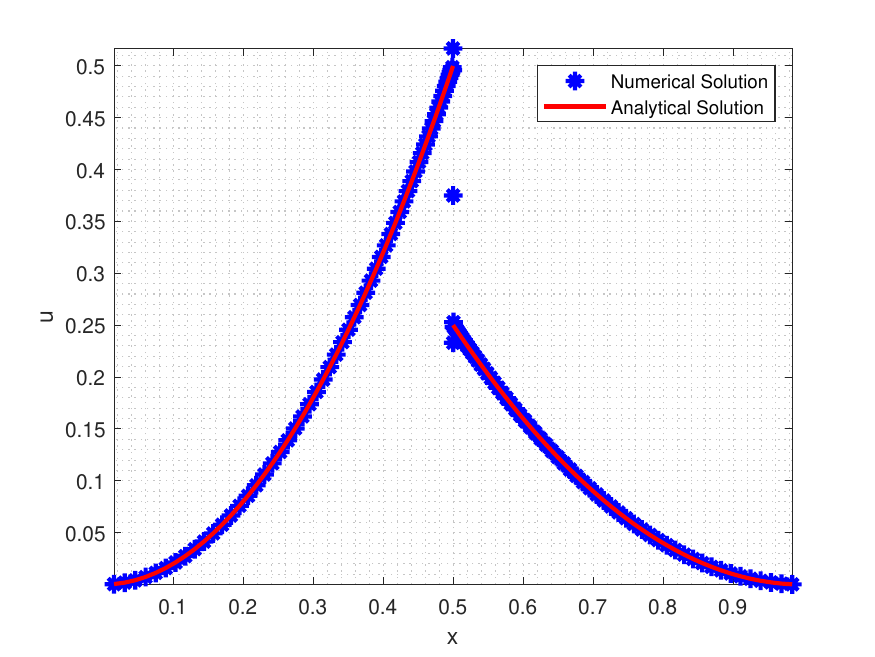}
    \centerline{(b)\,\,graded meshes with $\gamma =2$.}
	\end{minipage}
	\caption{ Analytical and numerical solutions for $\rho_\delta(s)=\frac{3}{\delta^3}$.} \label{fig44}
\end{figure}

 It should be pointed out that no instability has been observed in any of our computations. The piecewise linear FEM we use is not only convergent to the nonlocal constrained value problem with fixed $\delta$ but also convergent to the local differential equation with a fixed ratio between $\delta$ and $h$. Thus, our scheme is an asymptotically compatible scheme \cite{Tian2014}. In addition, we find from Figure \ref{fig44} that the numerical solutions of FEM on graded meshes have a better approximation effect and stability than the FEM on uniform meshes at the discontinuity point.

\medskip
\noindent {\bf Example 3.} {\bf (Limiting behaviours).}
 It is interesting to numerically study the local limit (i.e., $\delta \to 0$)  and the nonlocal interactions approach infinite (i.e., $\delta \to \infty$) of the nonlocal BVP \eqref{problem32} with $\lambda=0$. 
More specifically, we compare solutions of the nonlocal BVP \eqref{problem32} with $\lambda=0$ to the ODE  (resp. fractional Laplacian equation) equation with special interest in observing behaviors in the local limit $\delta \to 0$ (resp. fractional limit $\delta \to \infty$).
  
\medskip
\noindent{\bf Case 1 (local limit):} To study the local limit (i.e., $\delta \to 0$), we adopt the piecewise linear FEM for solving the nonlocal BVP \eqref{problem32} with $\lambda=0$, source function $f(x)=2$ and kernel function \eqref{fracker0}.
The exact solution is unknown, but the following ODE
 \begin{equation}\label{examp4-2}
\begin{cases}
  -u''(x)= 2,\;& \text{for}\ x\in (-1,1),\\[4pt]
u(-1)=u(1)= 0,
\end{cases}
\end{equation}
has the exact solution $u(x)=1-x^2$.

In our computation, we use the graded meshes \eqref{GradedM1}. 
We compare solutions of the nonlocal model to the ODE \eqref{examp4-2} with special interest in observing behaviors in the local limit $\delta \to 0$.  Figure \ref{Figdelta1} shows the comparison of numerical solutions for the nonlocal model and the solution of the governing equation given by \eqref{examp4-2} for $\gamma =2 $ and $ \gamma =3 $, respectively. It can be clearly observed that as $\delta \to 0$, the solution of the nonlocal diffusion model converges to the solution of the local one \eqref{examp4-2}. 


\begin{figure}
	\begin{minipage}[t]{0.48\linewidth}
		\centering
		\includegraphics[width=0.96\textwidth]{ 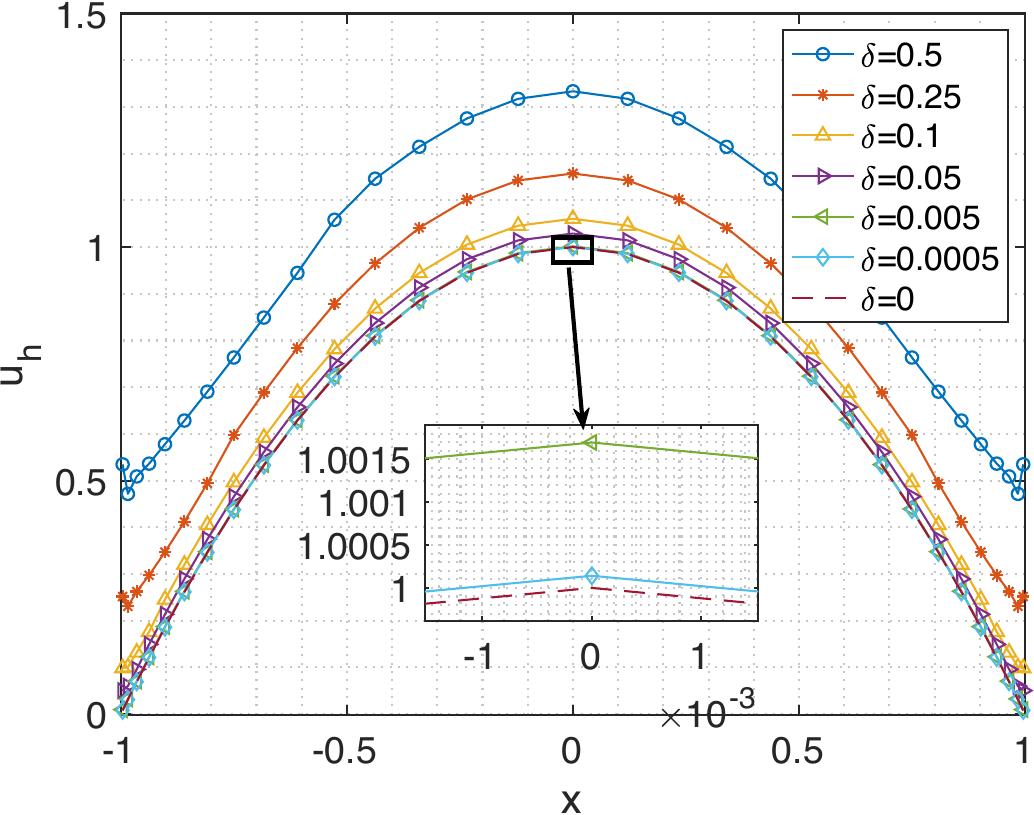}
       \centerline{(a)\,\, $\gamma =2$.}
	\end{minipage}%
	\begin{minipage}[t]{0.48\linewidth}
		\centering
		\includegraphics[width=0.96\textwidth]{ 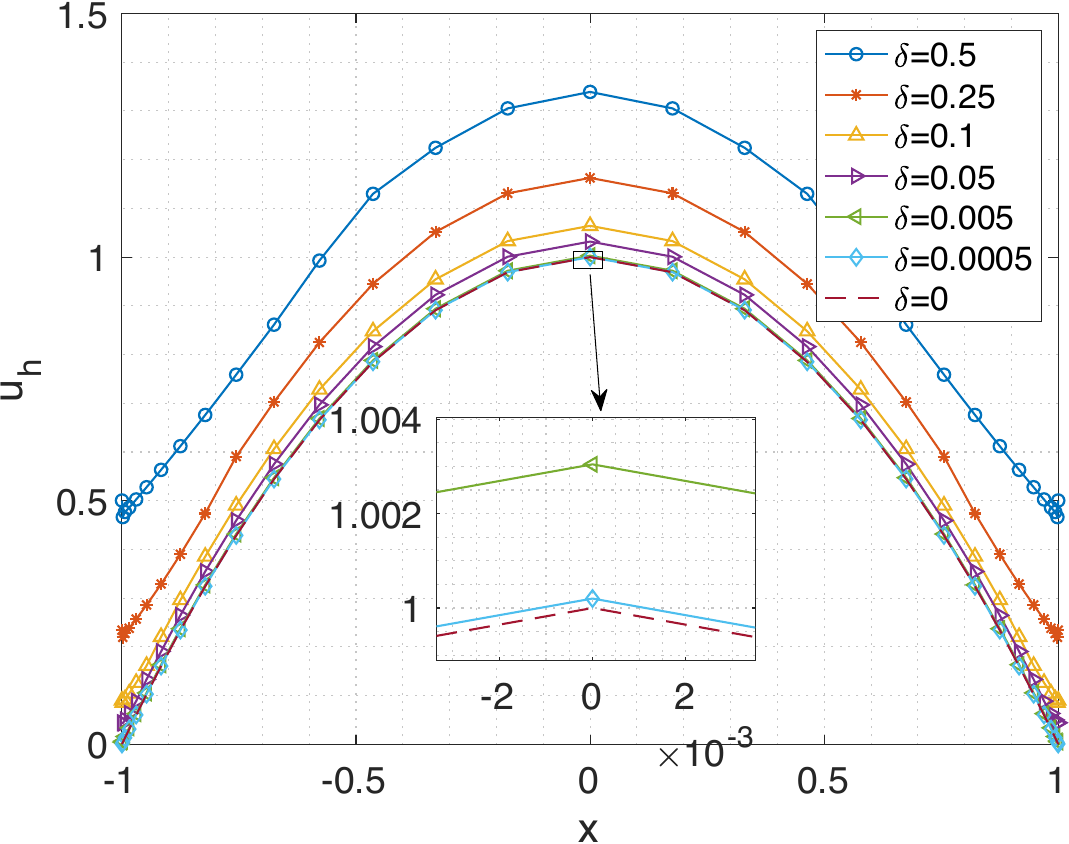}
       \centerline{(b)\,\, $\gamma =3$.}
	\end{minipage}

	\centering
	\begin{minipage}{0.48\linewidth}
		\centering
		\includegraphics[width=1.06\linewidth]{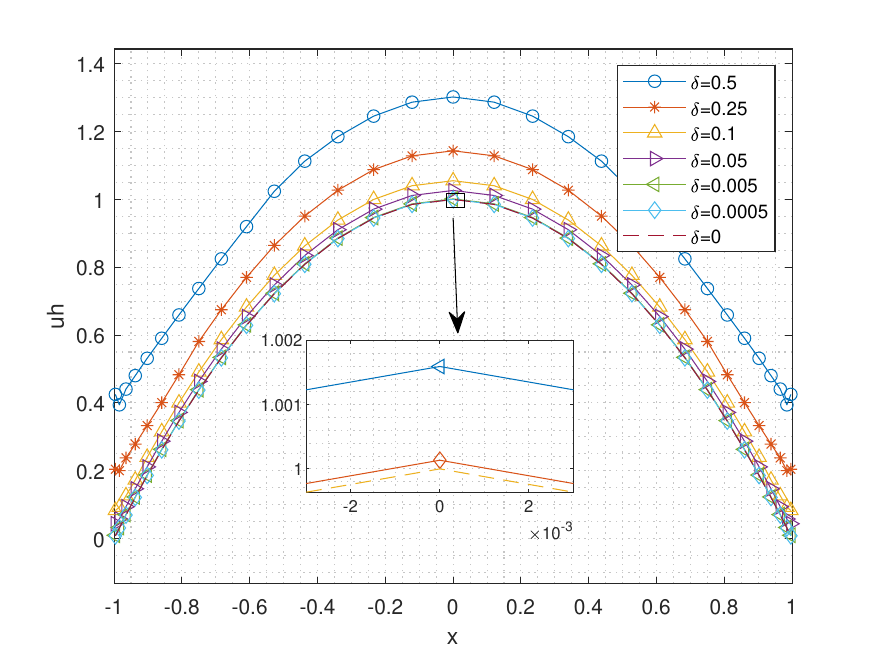}
    \centerline{(c)\,\, $\gamma =2$.}
		\label{chutian1}
	\end{minipage}
	\begin{minipage}{0.48\linewidth}
		\centering
		\includegraphics[width=1.06\linewidth]{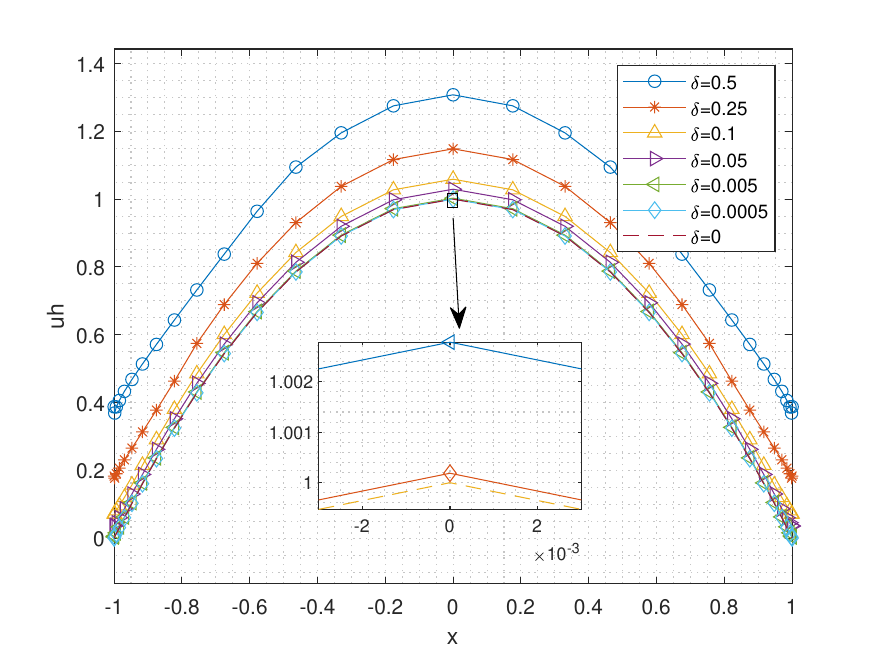}
    \centerline{(d)\,\, $\gamma =3$.}
		\label{chutian2}
	\end{minipage}

    \caption{The evolution process of the numerical solutions on graded meshes as $\delta \rightarrow 0$. Top: $\rho_\delta(s)=\frac{3}{\delta^3}$, Bottom: $\rho_{\delta}(s) = \frac{2-\alpha}{\delta^{2-\alpha}}s^{-(1+\alpha)}$ and $\alpha=-0.25$.}
    \label{Figdelta1}
\end{figure}

\medskip
\noindent
{\bf Case 2 (global limit):} Now, we turn to the global limit (i.e., $\delta\rightarrow\infty$), that is, consider the nonlocal models with the following fractional Laplacian kernel \eqref{fracker0}. 
It is known that the solution of the nonlocal problem with a fractional Laplacian kernel exhibits singularities near the boundary of $\Omega$. We take $\Omega=(-1,1)$ and consider a piecewise linear FEM on graded meshes \eqref{GradedM1}  for solving the nonlocal BVP \eqref{problem32} with $\lambda=0$ and $f(x)=1$.
The exact solution is unknown, but the following fractional Poisson equation
 \begin{equation}\label{examp5-2}
\begin{cases}
  (-\Delta)^{\alpha} u(x) = 1,\quad  & x\in \Omega= (-1,1),\\[4pt]
   u(x)=0, \quad & x\in\Omega^c=\mathbb{R}\backslash\bar{\Omega},
\end{cases}
\end{equation}
where
 \begin{equation}\label{fracLap-defn}
(-\Delta)^{\alpha} u(x)=C_{\alpha}\, {\rm p.v.}\! \int_{\mathbb R} \frac{u(x)-u(y)}{|x-y|^{1+2\alpha}} \rd y, \;\;\;\;
\end{equation}
with ``p.v." stands for the principle value. The equation \eqref{examp5-2} admits a exact solution $ u(x)=(1-x^{2})^{\alpha}_+/{\Gamma (2\alpha+1)}$.

As we can see from Figure \ref{fig5}, the solution of the nonlocal BVP converges to the solution of the fractional Laplacian equation \eqref{examp5-2} as the nonlocal interactions become infinite, i.e., $\delta \to \infty$. Thus the conclusions in \cite{Elia2013} are verified numerically.

\begin{figure}
	\begin{minipage}[t]{0.5\linewidth}
		\centering
		\includegraphics[width=1.06\textwidth]{ 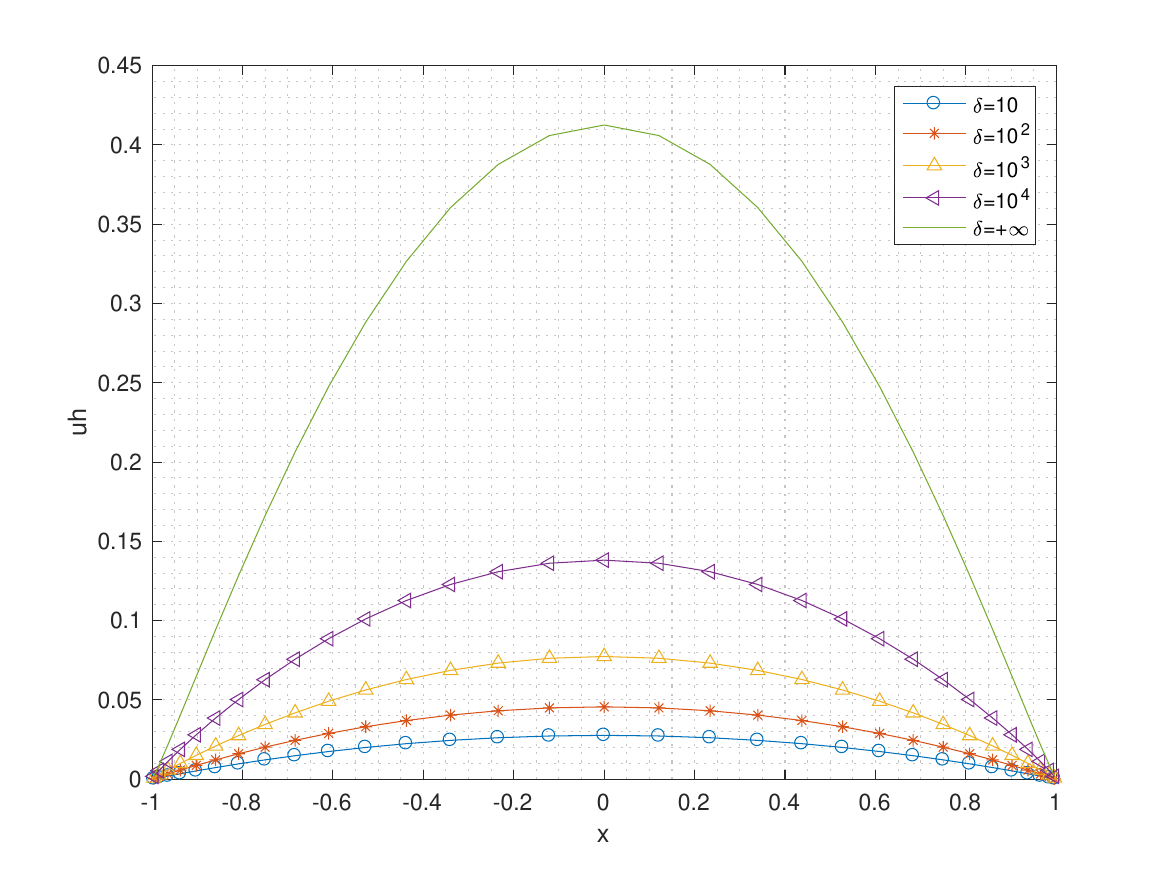}
       \centerline{(a)\,\, $\gamma =2$.}
	\end{minipage}%
	\begin{minipage}[t]{0.5\linewidth}
		\centering
		\includegraphics[width=1.06\textwidth]{ 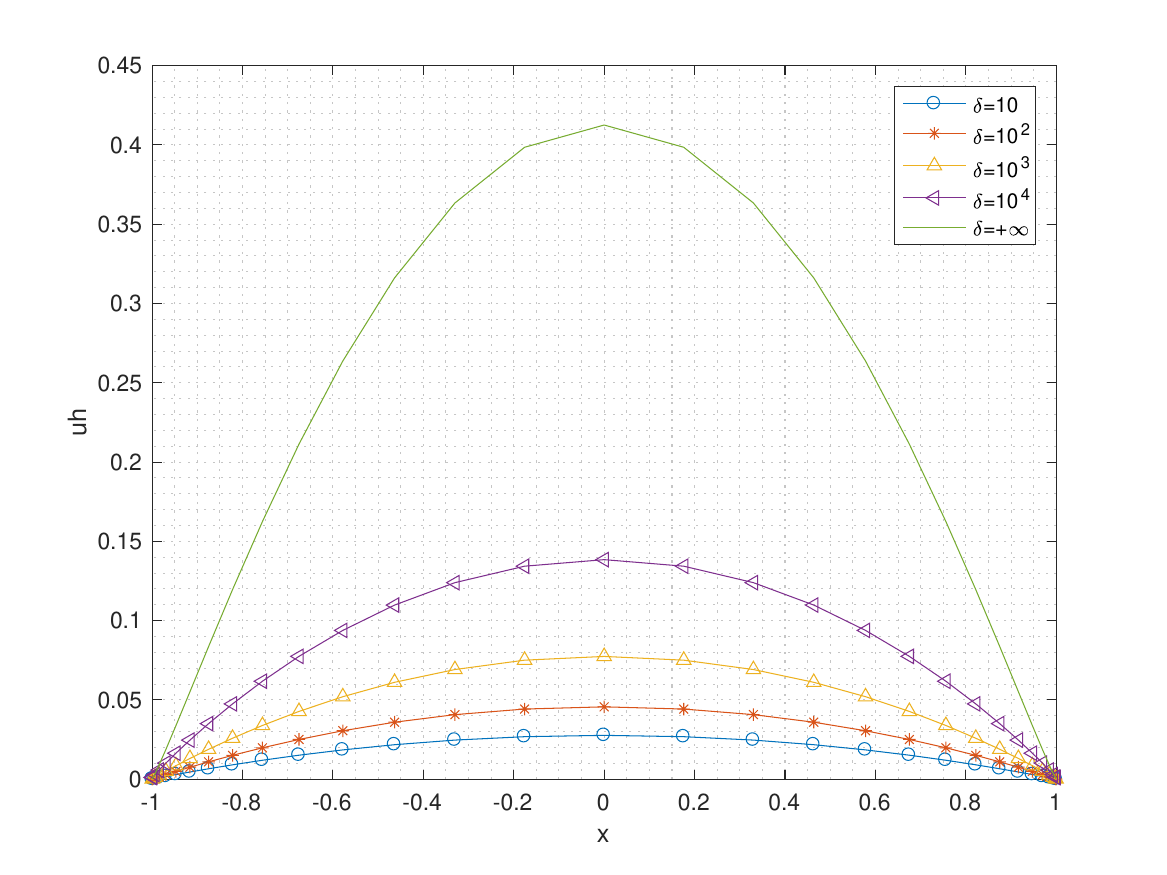}
        \centerline{(b)\,\, $\gamma =3$.}
	\end{minipage}
	\caption{The evolution process of the numerical solutions on graded meshes with $\gamma=1.1$ as $\delta\rightarrow\infty$. }
   \label{fig5}
\end{figure}

\subsubsection{\bf Nonlocal BVP  with $\lambda<0$}
In this subsection, we will consider the dynamics of nonlocal BVPs \eqref{problem32} with $\lambda<0$.

\medskip
\noindent {\bf Example 4.} {\bf (The dynamics of nonlocal Helmholtz problem).} In this example, we will consider the FEM on uniform meshes for the following nonlocal Helmholtz problem
\begin{equation}\label{Helm}
\begin{cases}
    \mathcal{L}_\delta u(x)-k^2n(x)u(x)=f(x),\;&\text{on}\;\Omega =(-L, L),\\
    u(x)=0,&\text{on}\;\Omega_\delta=[-L-\delta, L+\delta].
\end{cases}
\end{equation}
where $n(x)$ is a given function and $k^2$ is a given constant. Here, we focus on the dynamics of the solution with the following two cases: 

{\bf Case 1:} $n(x)$ is continuous function
\begin{equation}
 n(x) = \sin{x},\quad x\in (-L, L).
\end{equation}

{\bf Case 2:} $n(x)$ is piecewise constant
\begin{equation}
 n(x) =
  \begin{cases}
  0.5, \;  & x\in(-L,0),\\[1pt]
  1, \; & x\in(0,L).
\end{cases}
\end{equation}
For the convenience of description, we take $f(x)=k^2$ and the kernel function \eqref{fracker0}, i.e., $\rho_{\delta}(s) = \frac{2-\alpha}{\delta^{2-\alpha}}s^{-(1+\alpha)}$ with fixed  $\alpha=0.5$.

In Figure \ref{FigHelm_sin}, we depict the profile of the numerical solution for {\bf Case 1} on the computation domain $\Omega=[-4\pi,4\pi]$ with different $k^2$ and $\delta$. We observe from the top of Figure \ref{FigHelm_sin} that oscillation in regions with $\sin(x)<0$ exponential growth and decay in regions with $\sin(x)>0$. This observation is consistent with the one described in \cite[p.81]{Trefethen2018ODE} for the local model. Moreover, we find from the bottom of Figure \ref{FigHelm_sin} that as $k^2$ decreases, the oscillation frequency within the oscillatory region amplifies.

\begin{figure}[htbp]
  \centering
  \subfigure[$\delta = 0.1,\, k^2=1000/3$]{
    \includegraphics[height=3.5cm,width=0.32\columnwidth]{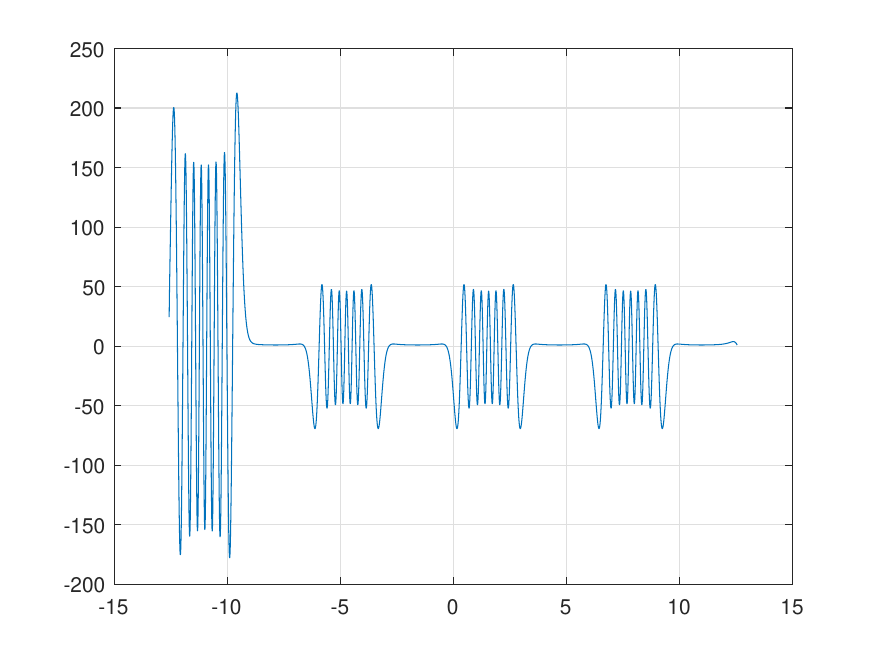}}
  \subfigure[$\delta = 0.001,\, k^2=1000/3$]{
    \includegraphics[height=3.5cm,width=0.32\columnwidth]{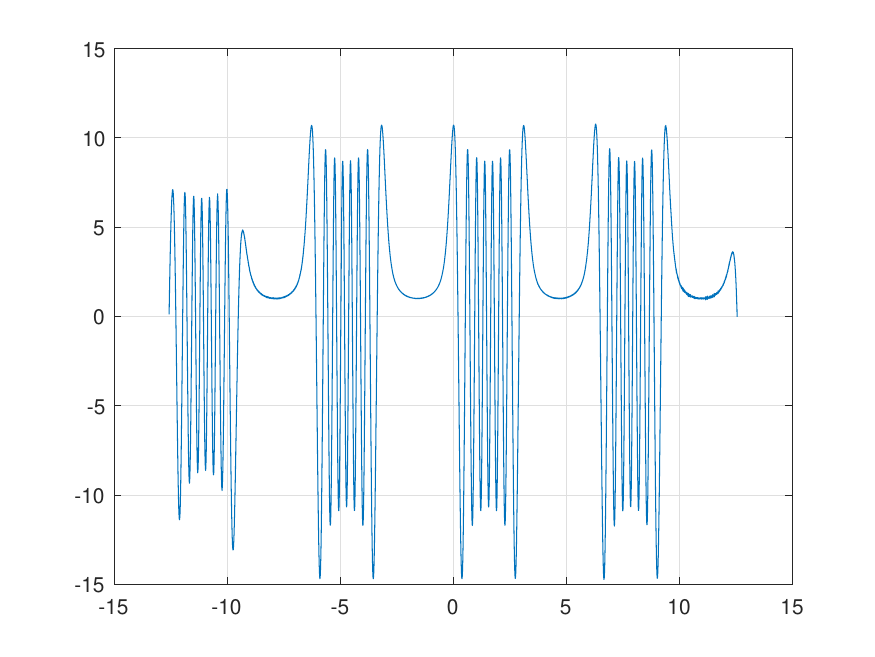}}
  \subfigure[$\delta = 0.0001,\, k^2=1000/3$]{
    \includegraphics[height=3.5cm,width=0.32\columnwidth]{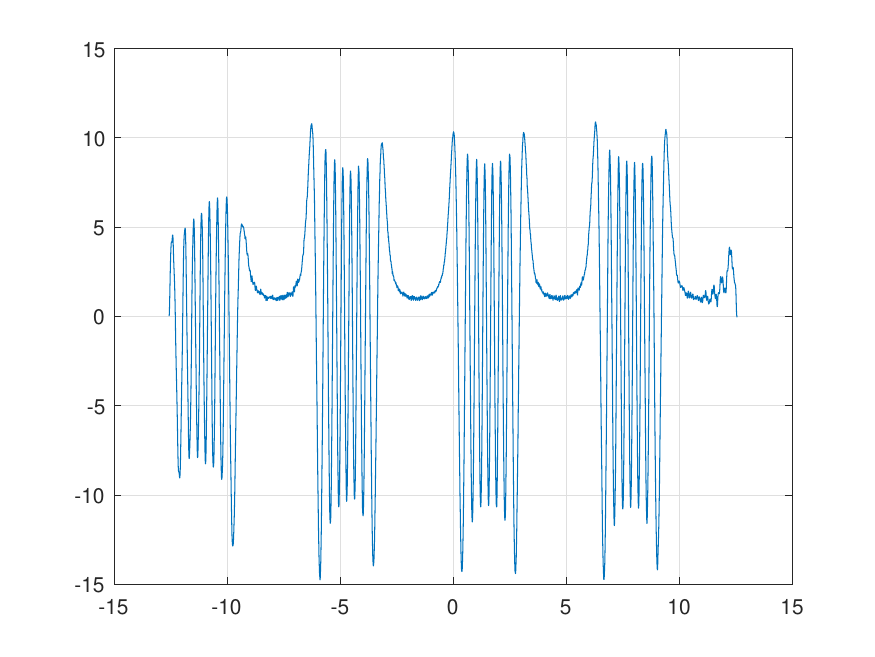}}

 \subfigure[$k^2= 10/3,\, \delta=0.001$]{
    \includegraphics[height=3.5cm,width=0.32\columnwidth]{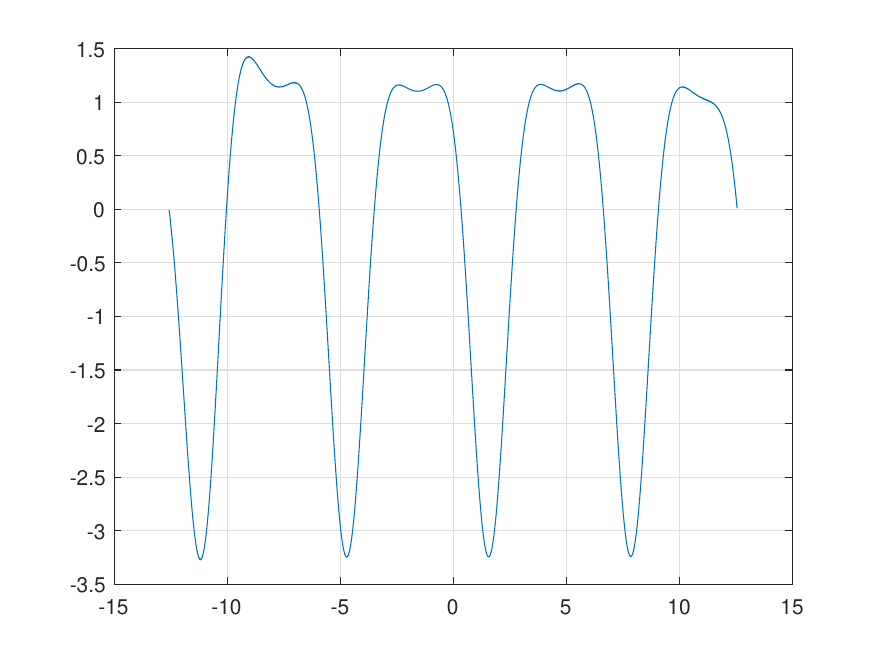}}
  \subfigure[$k^2 = 1000/3,\, \delta=0.001$]{
    \includegraphics[height=3.5cm,width=0.32\columnwidth]{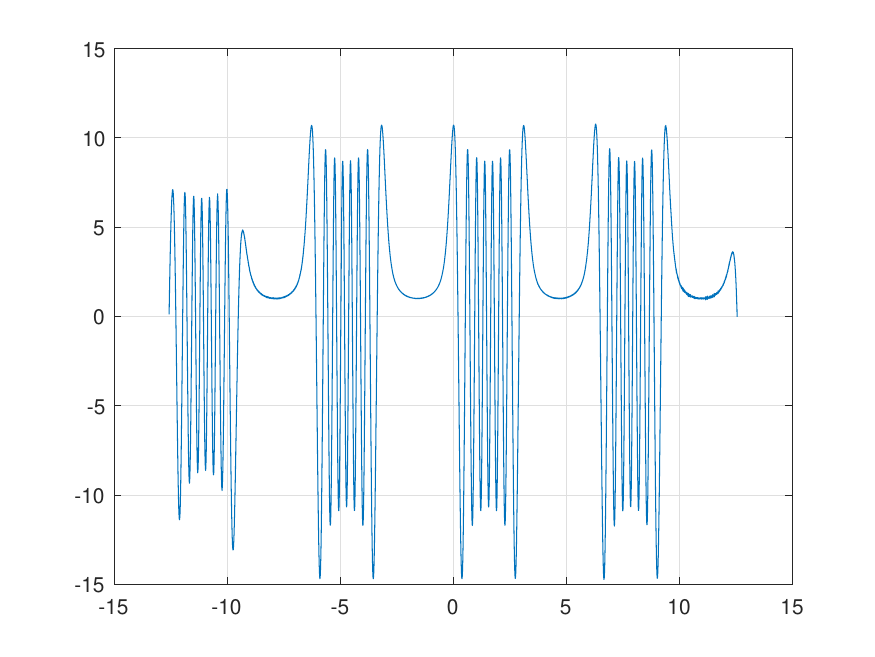}}
  \subfigure[$k^2 = 10000/3,\, \delta=0.001$]{
    \includegraphics[height=3.5cm,width=0.32\columnwidth]{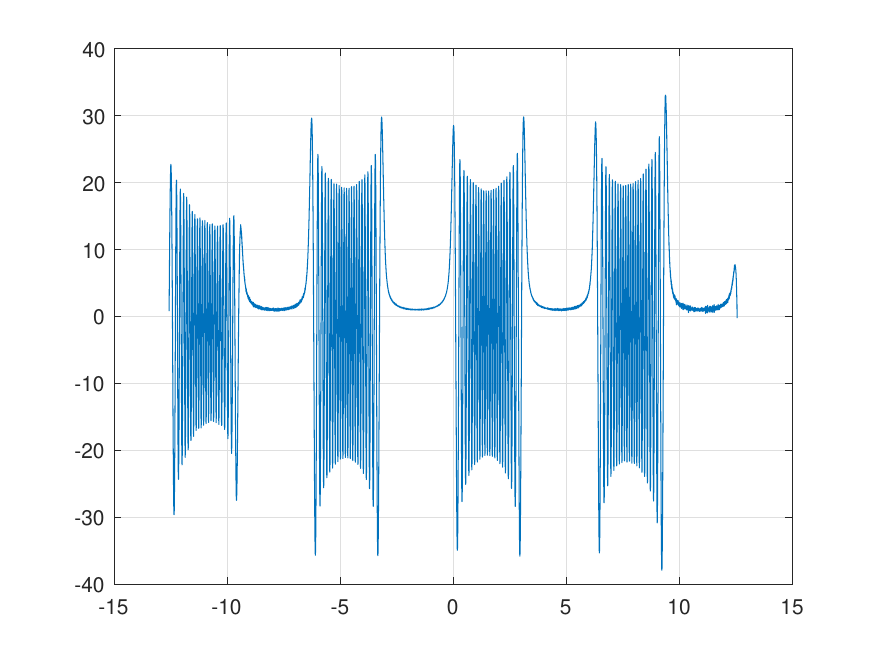}}

  \caption{The profiles of numerical solutions for Case 1. Top: versus different $\delta$ with fixed $k^2=1000/3$; Bottom: versus different $k^2$ with fixed $\delta=0.001$. }\label{FigHelm_sin}
\end{figure}



In Figure \ref{FigHelm_piecewise}, we plot the profile of the numerical solution for {\bf Case 2} on the computation domain $\Omega=[-100,100]$ with $k^2=2$. It can be seen that the amplitude of the solution is affected by magnitude of the piecewise constant function $n(x)$, that is,  the amplitude is larger within the interval $(-100, 0)$, while it is relatively smaller within the interval $(0, 100)$. We also observe that as $\delta$ decreases, the solution converges to the solution of the local Helmholtz problem, i.e., $\delta=0$.

\begin{figure}[htbp]
  \centering
  \subfigure[$\delta = 0.1,\, k^2=2$]{
    \includegraphics[height=3.5cm,width=0.32\columnwidth]{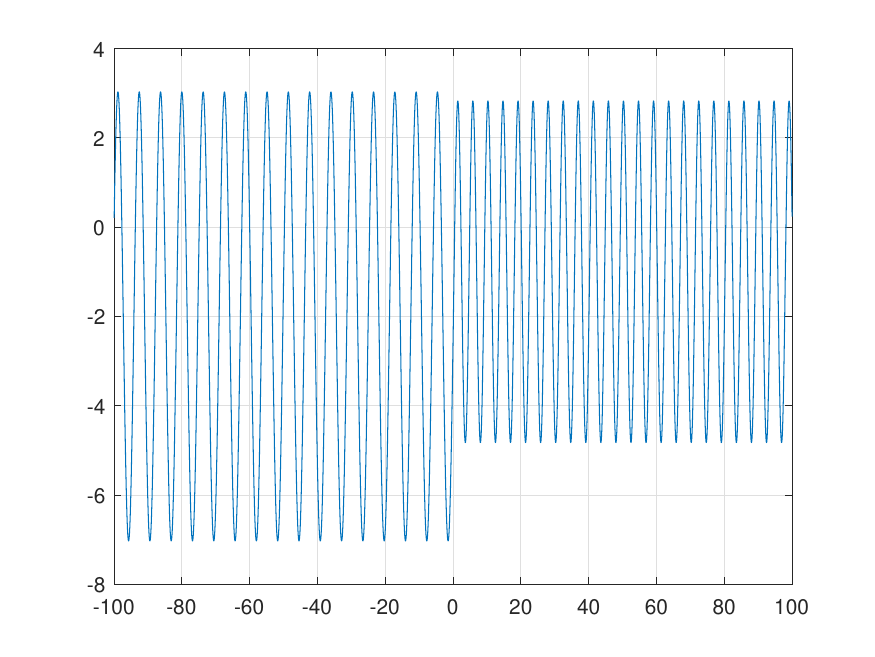}}
  \subfigure[$\delta = 0.01,\, k^2=2$]{
    \includegraphics[height=3.5cm,width=0.32\columnwidth]{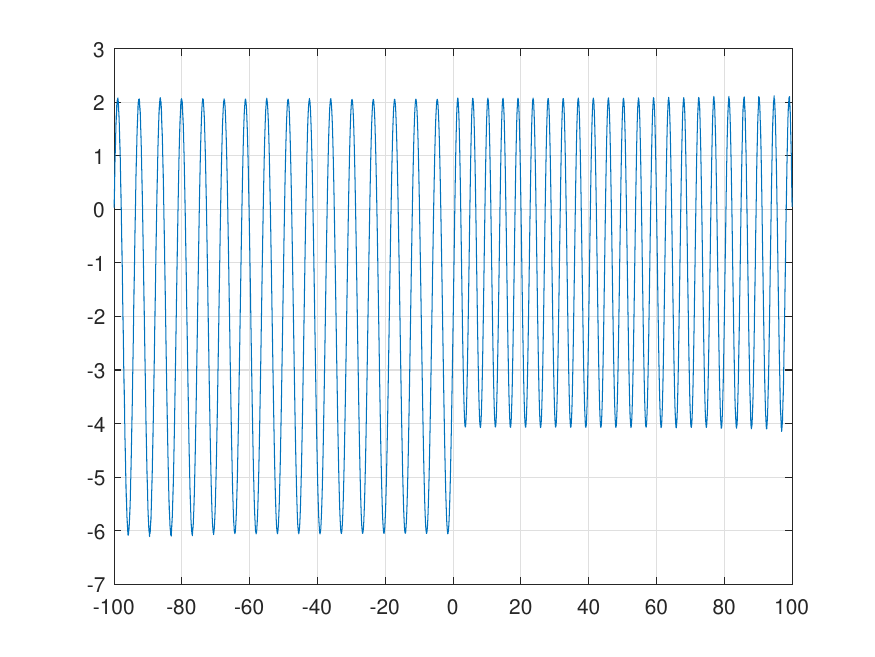}}
  \subfigure[$\delta = 0,\, k^2=2$]{
    \includegraphics[height=3.5cm,width=0.32\columnwidth]{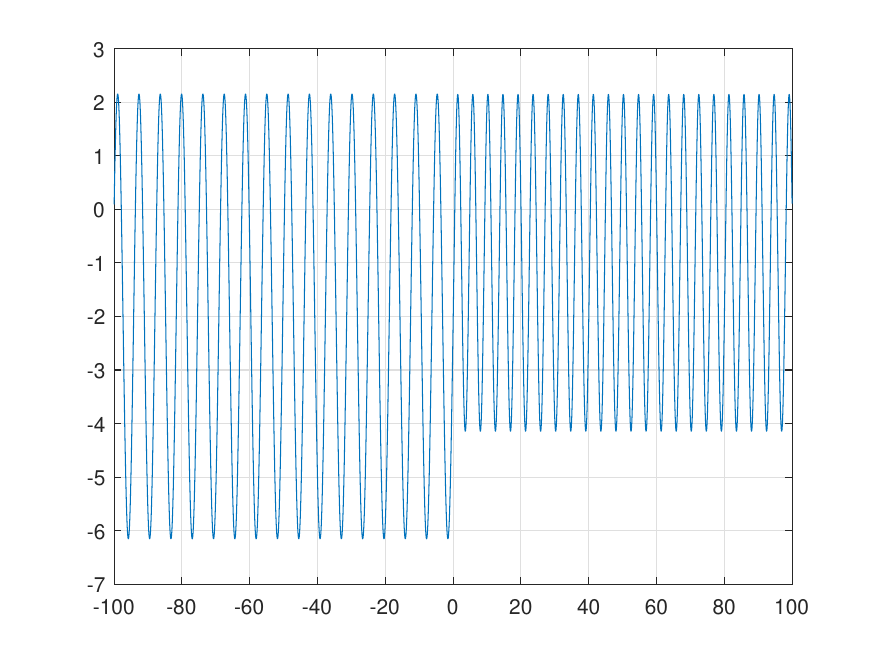}}

  \caption{The profiles of numerical solution for Case 2  with 
  $\alpha=0.5$ and $k^2=2$.}\label{FigHelm_piecewise}
\end{figure}

 \subsection{Time-dependent problem} 
 Consider the following nonlocal Allen-Cahn equation:
\begin{equation}\label{NAC}
\begin{cases}
 u_t(x,t)  + \epsilon^2 \mathcal{L}_{\delta} u(x,t) + f(u) = 0,\quad & \text{on}\ \Omega = (a, b),\quad 0<t\leq T,\\[4pt]
u(x,t) = 0,\quad & \text{on}\ \Omega_{\mathcal{\delta}} = [a-\delta, a] \cup [b, b+\delta],\quad 0<t\leq T,\\[4pt]
u(x,0)=u_0(x) ,\quad & \text{on}\ \Omega = (a, b).
\end{cases}
\end{equation}
where $u_0(x)$ is the initial data, $f(u)$ is given function on $\R$ satisfies $f(u)=F'(u)$ with $F(u)=\frac{1}{4}(u^2-1)^2$. In the above, $\mathcal{L}_{\delta}$ is the nonlocal operator \eqref{NonLOper} and $\epsilon$ is a constant.

Let $N$ be a positive integer and $\tau = T/N$ be a time step size. We can get the time grid $ t_n=n\tau, 0\leq n \leq N$ and use $u_h^n\in \mathcal{V}_h^0$ to approximate $u_h(x,t_n)$, $n=0,1,\cdots,N$. Using backward Euler method to approximate $u_t(x,t_n)$, we can get
\begin{equation*}
u_t(x,t_n)\approx \delta_tu_h^n = \frac{1}{\tau}(u^n_{h}-u^{n-1}_{h}),\quad x\in\Omega,\quad 1\leq n \leq N.
\end{equation*}

Here, we employ the fully discrete scheme with the semi-implicit scheme in time and FEM discretization in space for solving the problem \eqref{NAC}: find $u_h^n \in H^1(0,T;\mathcal{V}_h^{0}(\Omega))$, such that
\begin{equation}\label{SemiImplicitScheme}
\begin{cases}
  (\delta_tu_h^{n},v_{h}) +  \epsilon^2 \mathcal{A}_\delta\left(u_{h}^{n}, v_{h}\right) + (f(u_h^{n-1}),v_{h}) = 0,~\forall\, v_{h}\in \mathcal{V}_h^{0}(\Omega),\quad 1\leq n \leq N,\\
  (u_h^0,v_h) = (u_0(x),v_h),\quad \forall\, v_h\in \mathcal{V}_h^{0}(\Omega).
\end{cases}
\end{equation}
Then, it is easy to obtain the matrix form of \eqref{SemiImplicitScheme} as follows:
\begin{equation}\label{MatrixEquofSemiImpli}
  (\bs{M}  + \tau \epsilon^2\bs{S_{\delta}}) \bs{U^{n}} =  \bs{M} \bs{U^{n-1}} - \tau \bs{F^{n-1}},
\end{equation}
where $\bs{M}$ is the usual (tridiagonal) FEM mass matrix, the stiffness matrix $\bs{S_{\delta}}$ is defined in \eqref{stiffenty}, and $\bs{F^{n-1}}$ is defined as
\begin{equation*}
  \bs F^{n-1} = \int_{\Omega}f(u_h^{n-1})\bs \Phi\, \rd x,\;\;\;\bs \Phi = \big(\varphi_1(x),\varphi_2(x),\cdots,\varphi_{N-1}(x)\big)^{\top}.
\end{equation*}

In this example, we adopt the proposed fully discrete scheme \eqref{SemiImplicitScheme} to simulate the phase evolution behavior and energy dissipation. We take $\Omega = (-1,1)$, $\epsilon = 0.01$, the initial date $u_0(x)= e^{-100x^2}$, and take the fractional kernel function of the form \eqref{fracker0}. 

\begin{figure}
	\begin{minipage}[t]{0.5\linewidth}
		\centering
		\includegraphics[width=0.96\textwidth]{ 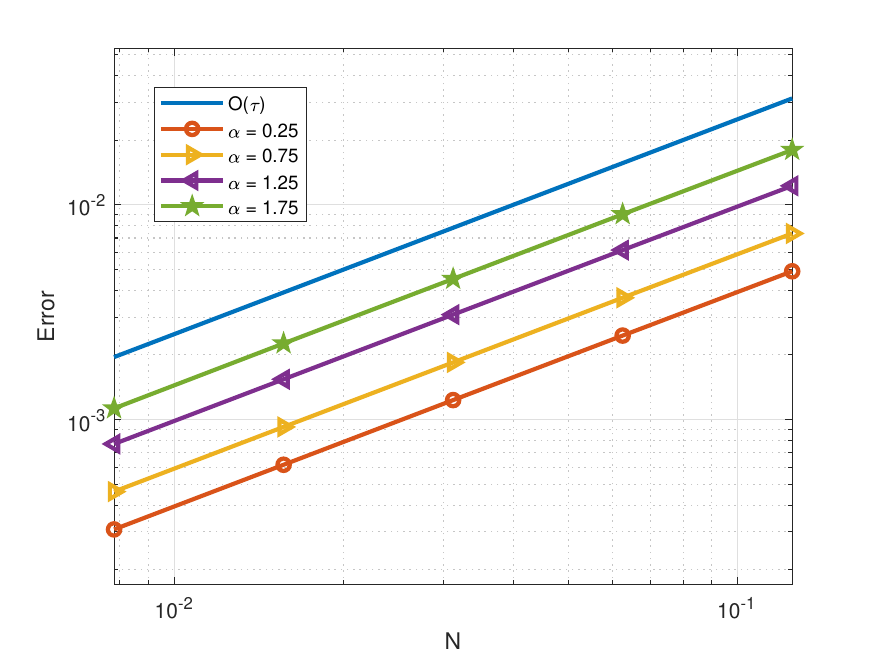}
		\centerline{(a)\,\,$\delta =0.01$.}
	\end{minipage}%
	\begin{minipage}[t]{0.5\linewidth}
		\centering
		\includegraphics[width=0.96\textwidth]{ 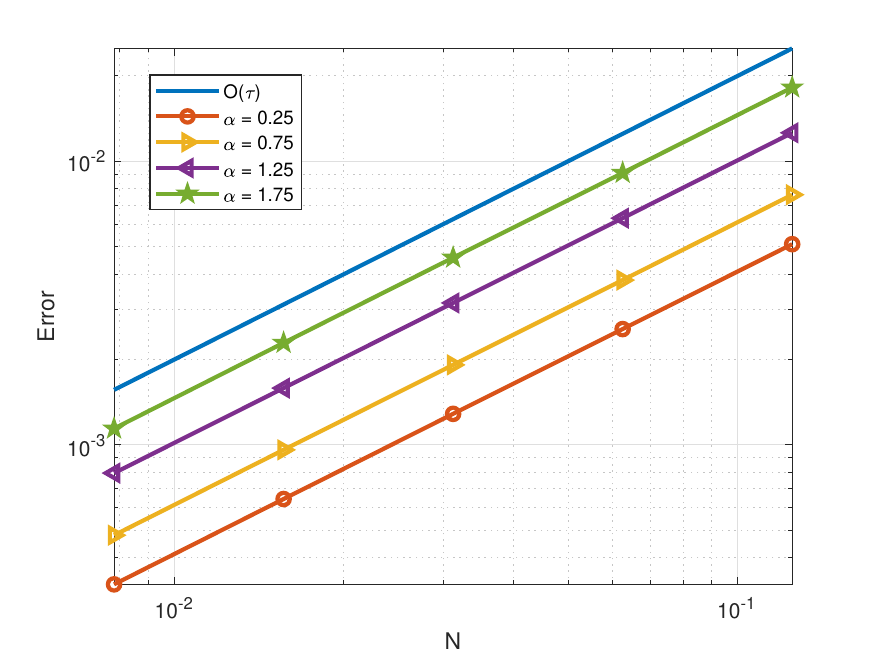}
	  \centerline{(b)\,\,$\delta =0.1$.}
	\end{minipage}
	\begin{minipage}[t]{0.5\linewidth}
		\centering
		\includegraphics[width=0.96\textwidth]{ 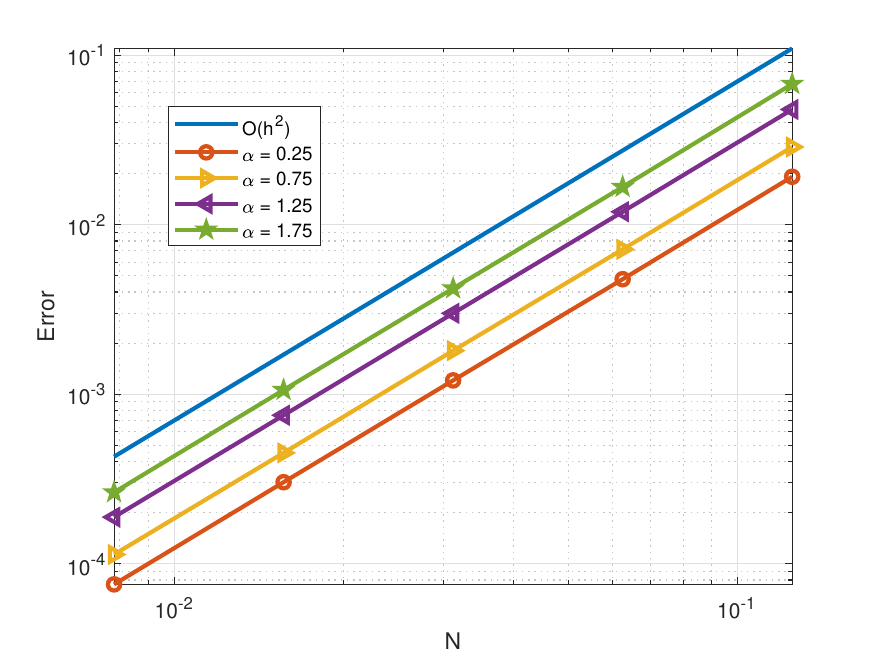}
		\centerline{(c)\,\,$\delta =0.01$.}
	\end{minipage}%
	\begin{minipage}[t]{0.5\linewidth}
		\centering
		\includegraphics[width=0.96\textwidth]{ 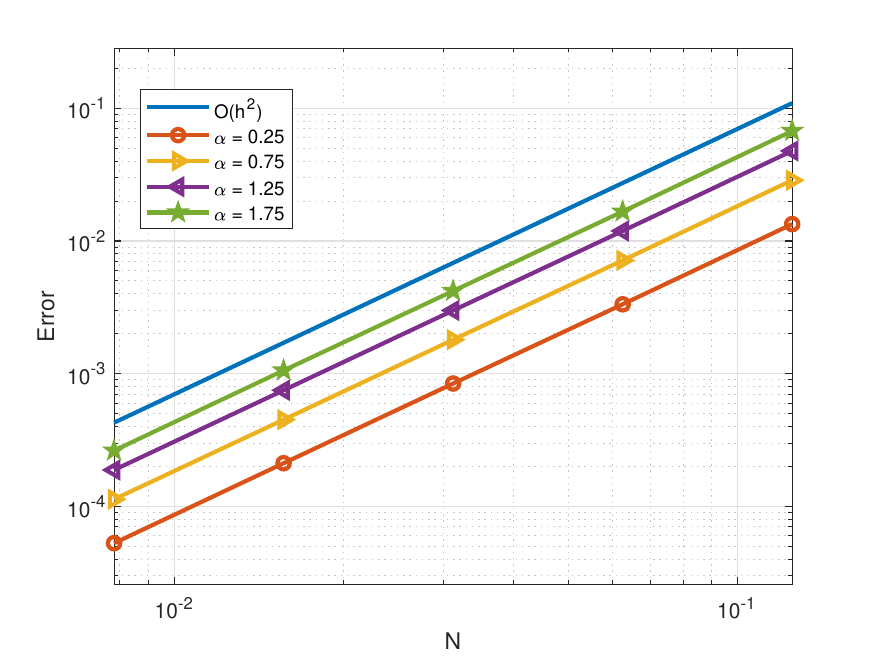}
		\centerline{(d)\,\,$\delta =0.1$.}
	\end{minipage}
	\caption{Convergence order of FEM with different $\delta$. (a)-(b) are the time convergence order, (c)-(d) are the spatial convergence order.}
    \label{Cov}
\end{figure}

In Figure \ref{Cov}, we plot the numerical errors and the corresponding convergence rates of the proposed method max errors for different $\alpha$ and $\delta$. As we can see the temporal convergence rate is $O(\tau)$, and the space convergence rate is $O(h^2)$.
Moreover, we display the waveform of the numerical solution at the different $\delta$ and time, see Figures \ref{Waveforms}. The dates indicate that the $\delta$ affects the propagation velocity of the solitary wave. For a smaller $\delta$, the propagation of the soliton becomes slower, thus indicating the presence of quantum sub-diffusion.
Finally, we plot the evolution of maximum value at various times with different $\alpha$ and $\delta$ in Figure \ref{Evofmax}, which shows that the maximum principle is preserved numerically. We observe from Figure \ref{Evofmax} (d) that the maximum value of the steady state is increased as $\alpha$ increases. 

\begin{figure}[htbp]
  \centering
  \subfigure[$\delta = 0.02$]{
  \includegraphics[height=5cm,width=0.45\columnwidth]{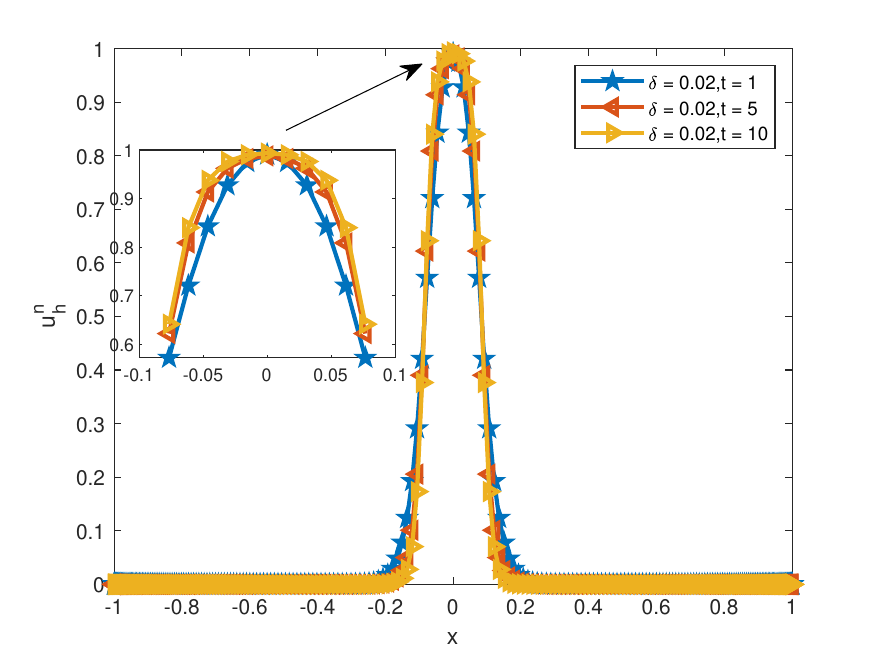}}
  \quad
  \subfigure[$\delta = 0.2$]{
  \includegraphics[height=5cm,width=0.45\columnwidth]{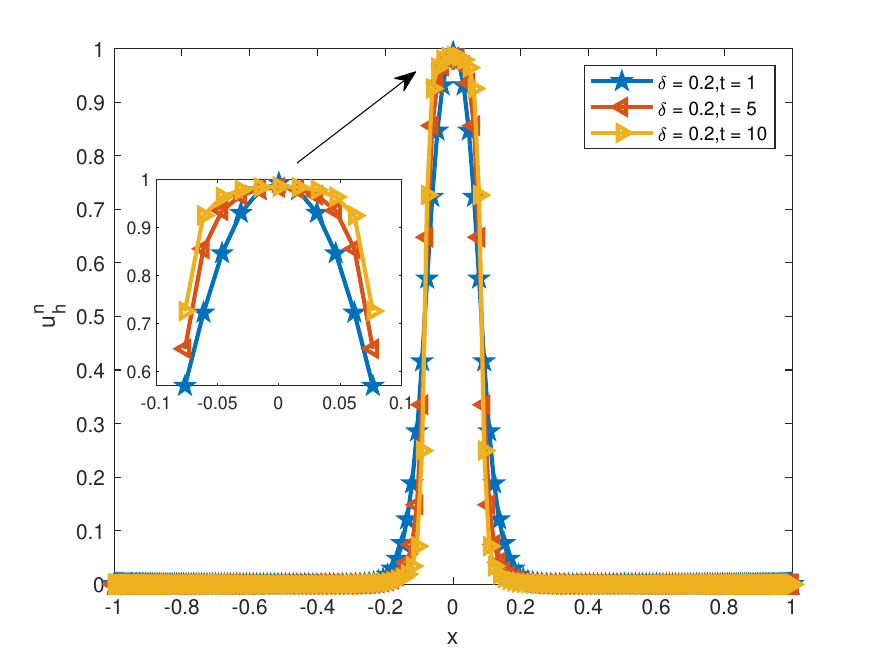}}
  \quad
  \subfigure[$\delta = 2$]{
  \includegraphics[height=5cm,width=0.45\columnwidth]{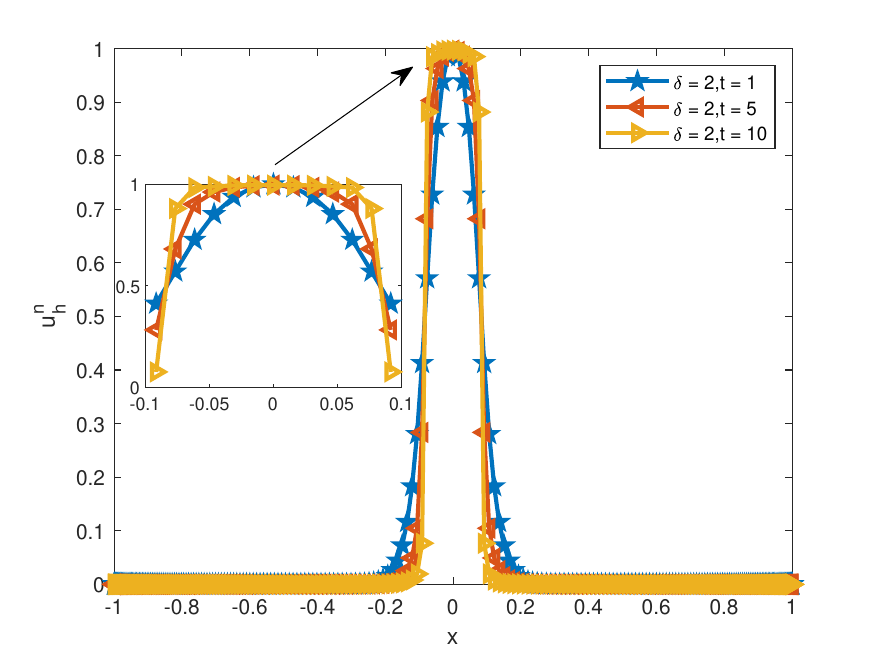}}  
  \quad
  \subfigure[$ t= 10$]{
  \includegraphics[height=5cm,width=0.45\columnwidth]{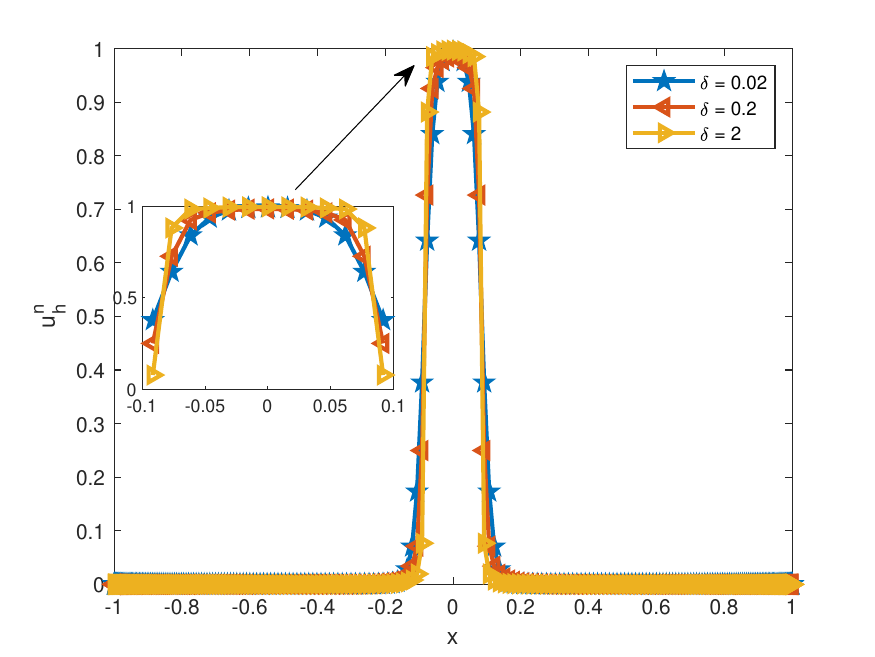}} 
  \caption{ The waveforms of the numerical solution with the $\alpha = 1.25$ on nonuniform meshes. (a)-(c) are the snapshots of $u(x)$ are taken at $t = 1, 5, 10$. (d) is the snapshots of $u(x)$ are taken at $t = 10$ for different $\delta$.}
  \label{Waveforms}
\end{figure}

\begin{figure}[htbp]
 \centering
 \begin{minipage}{0.48\linewidth}
  \centering
  \includegraphics[width=0.9\linewidth]{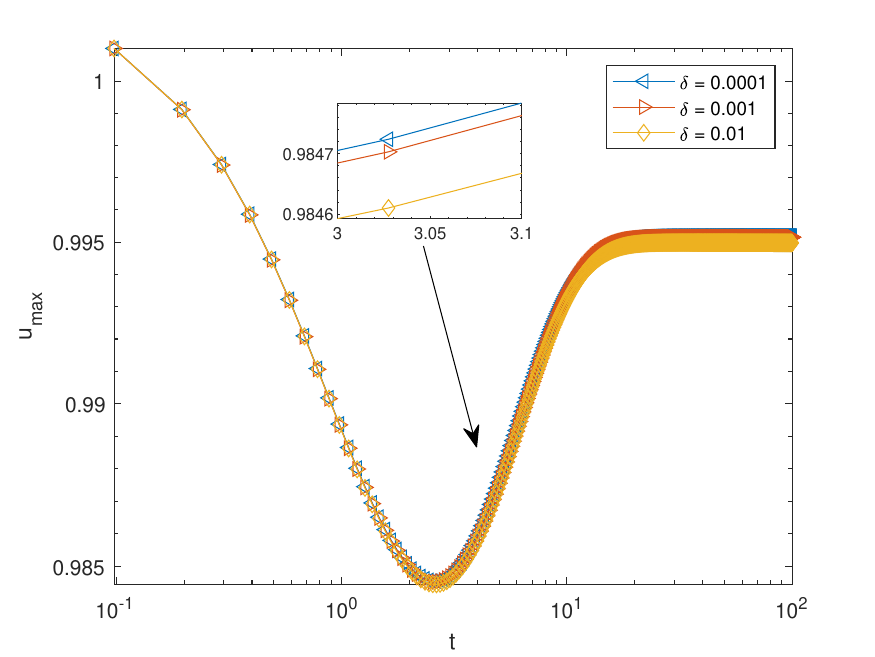}
  \centerline{(a)\,\,$\alpha = -0.25$.}
 \end{minipage}
 \begin{minipage}{0.48\linewidth}
  \centering
  \includegraphics[width=0.9\linewidth]{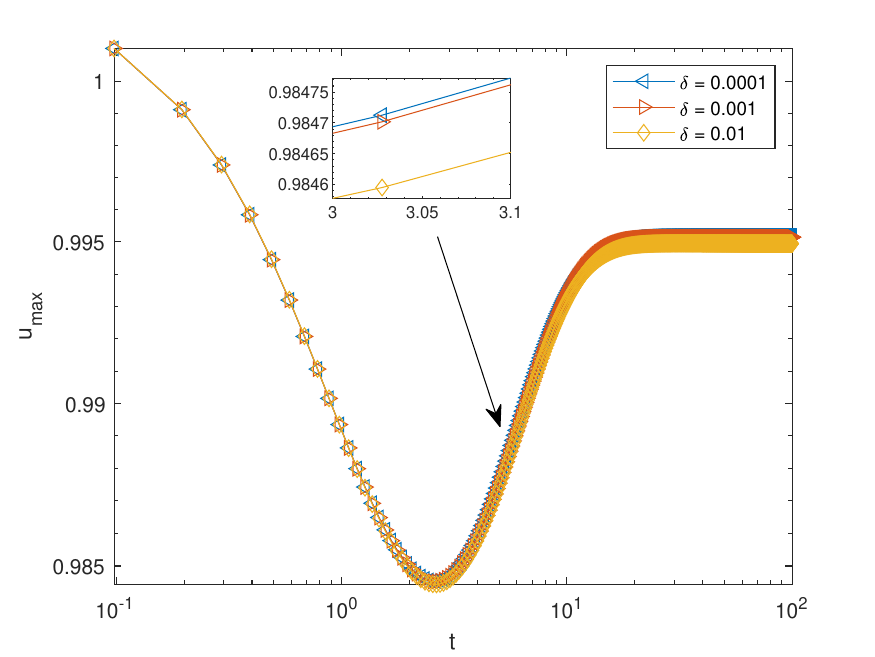}
  \centerline{(b)\,\,$\alpha = 0.25$.}
 \end{minipage}

 \begin{minipage}{0.48\linewidth}
  \centering
  \includegraphics[width=0.9\linewidth]{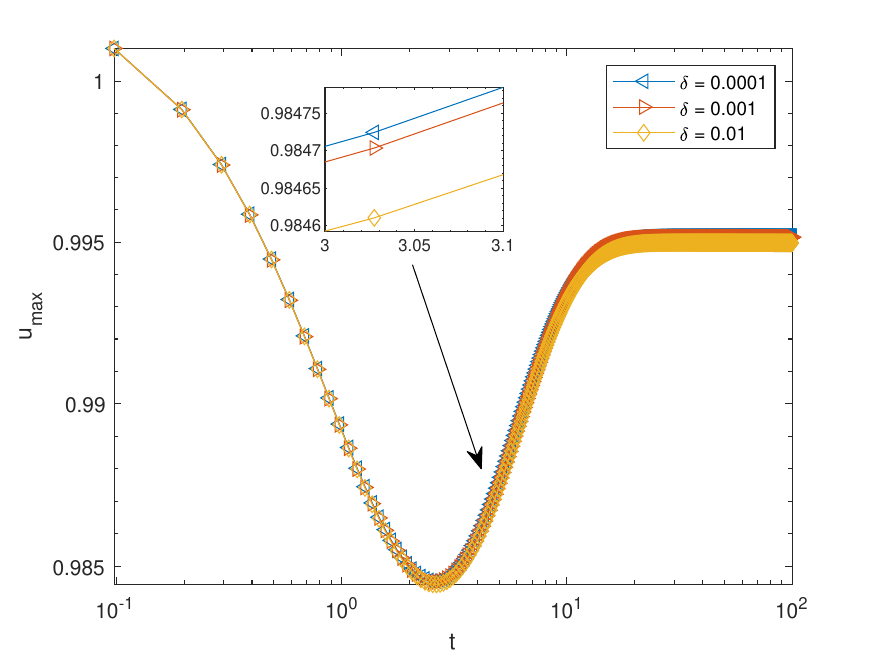}
  \centerline{(c)\,\,$\alpha = 0.75$.}
 \end{minipage}
 \begin{minipage}{0.48\linewidth}
  \centering
  \includegraphics[width=0.9\linewidth]{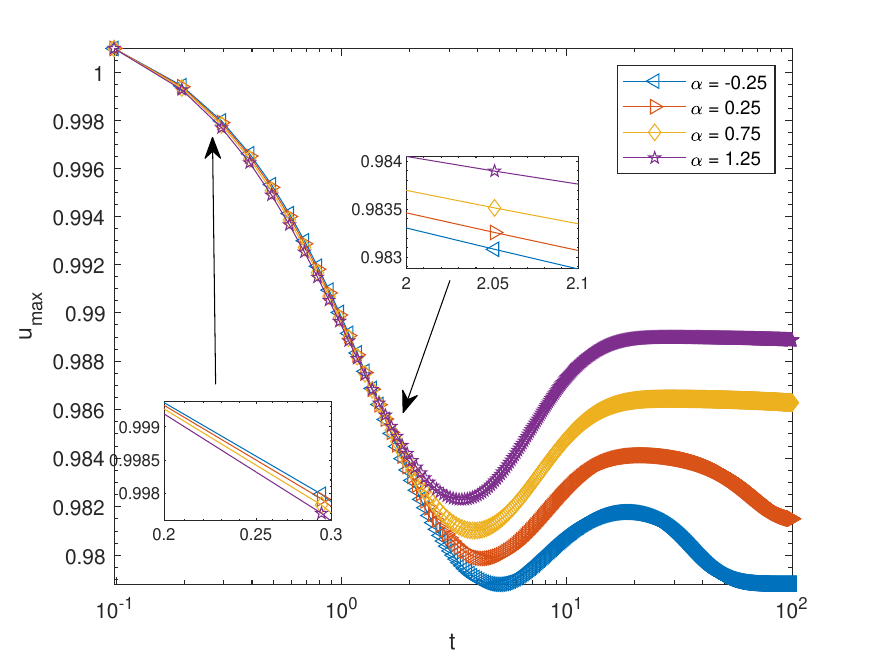}
        \centerline{(d)\,\,$\delta = 0.1$.}
 \end{minipage}

    \caption{Evolution of maximum value as $\delta$ increases on nonuniform meshes. (a)-(c) are the maximum value at $\alpha = -0.25, 0.25, 0.75$. (d) is the maximum value at $\delta=0.1$ for different $\alpha$.}
    \label{Evofmax}
\end{figure}

 \section{Concluding remarks and discussions}
 Different from the implementation of FEM in the physical space, we evaluated the nonlocal stiffness matrix of piecewise linear FEM for the nonlocal operator in the Fourier-transformed domain, and derived the explicit integral form of the entries on non-uniform meshes. This allowed us to adopt the different meshes for different situations, and enabled us to resort to the Gauss quadrature for one dimension integral accurately. Plenty of numerical examples for solving the nonlocal models are presented to demonstrate the accuracy and efficiency of the proposed method.

 The current study is largely based on a simple one-dimensional linear model for the sake of offering insight without being impeded by tedious calculations. More careful studies on the stability and convergence analysis of piecewise linear FEM on non-uniform meshes, and one may extend the study to two-dimensional rectangular elements, but it is  more involved, which we shall report in a future work.

\vspace{18pt}

\noindent {\LARGE \bf Declarations}

\vspace{12pt}

\noindent{\bf Conflict of interest}\;\;
 The authors declare that they have no conflict of interest.

\vspace{18pt}


\begin{thebibliography}{1}

\bibitem{AbramovitzS64}
M.~Abramovitz and I. A. Stegun,
\newblock {\em {Handbook of Mathematical Functions}},
\newblock Dover, New York, 1972.

\bibitem{Acosta2017}
G. Acosta and J. P. Borthagaray,
\newblock {\em {A fractional Laplace equation: regularity of solutions and finite element approximations}},
\newblock SIAM J. Numer. Anal., 55(2017)472--495.

\bibitem{Ainsworth2017}
 M. Ainsworth and C. Glusa,
\newblock {\em {Aspects of an adaptive finite element
  method for the fractional Laplacian: a priori and a posteriori error estimates, efficient implementation and multigrid solver}},
\newblock Comput. Methods Appl. Mech. Engrg., 327(2017)4-35.

 \bibitem{Akso2014}
B. Aksoylu and Z. Unlu,
\newblock {\em {Conditioning analysis of nonlocal integral operator in fractional Sobolev spaces}},
\newblock SIAM J. Numer. Anal., 52(2014)653-677.

\bibitem{askari2008peridynamics}
E. Askari, F. Bobaru,  R. B. Lehoucq, M. L. Parks, S. A. Silling and O. Weckner,
\newblock {\em {Peridynamics for multiscale materials modeling}},
 \newblock Journal of Physics: Conference Series, 2008.

\bibitem{Aulisa2022Efficient}
E. Aulisa, G. Capodaglio, A. Chierici, M. D'Elia, Marta,
\newblock {\em {Efficient quadrature rules for finite element discretizations of nonlocal equations}},
 \newblock Numer. Methods Partial Differential Equations, 38(2022)1767–1793.



 \bibitem{Bonito2018}
 A. Bonito, J. P. Borthagaray, R. H. Nochetto, E. Ot{\'a}rola and A. J. Salgado,
\newblock {\em {Numerical methods for fractional diffusion}},
\newblock Comput. Vis. Sci., 19 (2018)19-46.

 \bibitem{Cao2023}
R. Cao, M. Chen, Y. Qi, J. Shi, X. Yin, 
\newblock {\em {Analysis of (shifted) piecewise quadratic polynomial collocation for nonlocal diffusion model}},
\newblock Appl. Numer. Math. 185(2023)120–140.



 \bibitem{Chen2011}
X. Chen, M. Gunzburger,
\newblock {\em {Continuous and discontinuous finite element methods for a peridynamics model of mechanics}},
\newblock Comput. Methods Appl. Mech. Engrg., 200(2011)1237-1250.

\bibitem{Chen2021non-uniform}
H. Chen, C. Sheng, and L.-L Wang,
\newblock {\em {On explicit form of the FEM stiffness matrix for the integral fractional Laplacian on non-uniform meshes}},
\newblock Appl. Math. Lett., 113(2021)106864.

\bibitem{Gradshteyn2007Table}
 I. S. Gradshteyn and I. M. Ryzhik,
 \newblock {\em {Table of Integrals, Series and Products}}, 
 \newblock Seven Edition, Elsevier, 2007.

 \bibitem{Du2012}
Q. Du, M. Gunzburger, R. Lehoucq and K. Zhou,
\newblock {\em {Analysis and approximation of nonlocal diffusion problems with volume constraints}},
\newblock SIAM Review, 54(2012)667-696.

 \bibitem{Du2013}
Q. Du, M. Gunzburger, R. Lehoucq and K. Zhou,
\newblock {\em {A posteriori error analysis of finite element method for linear nonlocal diffusion and peridynamic models}},
\newblock Math. Comp., 82(2013)1889-1922.

 \bibitem{Du2019}
Q. Du,
\newblock {\em {Nonlocal Modeling, Analysis, and Computation}},
\newblock Society for Industrial and Applied Mathematics, Philadelphia, 2019.

 \bibitem{Du2023}
Q. Du, X. Tian and Z. Zhou
\newblock {\em {Nonlocal diffusion models with consistent local and fractional limits}},
\newblock A³N²M: Approximation, Applications, and Analysis of Nonlocal, Nonlinear Models, The IMA Volumes in Mathematics and its Applications, Springer, Cham. 165(2023)175–213.

\bibitem{Duo2018}
S. Duo, H. van Wyk and Y. Zhang,
\newblock {\em {A novel and accurate finite difference method for the fractional {L}aplacian and the fractional poisson problem}}, 
\newblock J. Comput. Phys., 355(2018)233-252.

 \bibitem{Elia2013}
M. D’Elia, M. Gunzburger,
\newblock {\em {The fractional Laplacian operator on bounded domains as a special case of the nonlocal diffusion operator}},
\newblock Comput. Math. Appl., 66(2013)1245– 1260.

\bibitem{Elia2020}
 M. D'Elia, Q. Du, C. Glusa, M. Gunzburger, X. Tian, and Z. Zhou, 
 \newblock {\em {Numerical methods for nonlocal and fractional models}}, 
 \newblock Acta Numer., 29(2020)1-124.


\bibitem{Di2012}
 E. Di Nezza, G. Palatucci, and E. Valdinoci, 
\newblock {\em {Hitchhiker's guide to the fractional Sobolev spaces}},
\newblock Bull. Sci. Math. 136(2012)521–573.

\bibitem{Fri1972}
 I. Fried, 
\newblock {\em {Condition of finite element matrices generated from nonuniform meshes}},
\newblock AIAA J. 10(1972)219-221.


\bibitem{Guan2022}
Q. Guan, M. Gunzburger, X. Zhang,
\newblock {\em {Collocation method for one dimensional nonlocal diffusion equations}},
\newblock  Numer. Methods Partial Differential Equations, 38(2022)1618–1635.



 \bibitem{Hao2021}
Z. Hao, Z. Zhang and R. Du,
\newblock {\em {Fractional centered difference scheme for high-dimensional integral fractional Laplacian}},
\newblock J. Comput. Phys., 424(2021)109851.



 \bibitem{Huang2014}
Y. Huang, A. Oberman,
\newblock {\em {Numerical methods for the fractional Laplacian: a finite difference-quadrature approach}},
\newblock SIAM J. Numer. Anal., 52(2014)3056-3084.

\bibitem{Klar2023}
M. Klar, G. Capodaglio, M. D'Elia, C. Glusa, M. Gunzburger, Max, and C. Vollmann, 
\newblock {\em {A scalable domain decomposition method for FEM discretizations of nonlocal equations of integrable and fractional type}},
\newblock Comput. Math. Appl. 151(2023)434-448.


\bibitem{Leng2021}
Y. Leng, X. Tian,  N. Trask, J. T. Foster, 
\newblock {\em {Asymptotically compatible reproducing kernel collocation and meshfree integration for nonlocal diffusion}},
\newblock SIAM J. Numer. Anal. 59(2021)88–118.



 \bibitem{Li2021}
H. Li, R. Liu and L.-L. Wang,
\newblock {\em {Efficient Hermite spectral-Galerkin methods for nonlocal diffusion equations in unbounded domains}},
\newblock Numer. Math. Theo. Meth. Appl., 15 (2022)1009-1040.

 \bibitem{Lischke2020}
A. Lischke, G. Pang, M. Gulian, and et~al., 
\newblock {\em {What is the fractional Laplacian? A comparative review with new results}}, 
\newblock J. Comput. Phys., 404(2020)109009.

 \bibitem{Liu2017}
Z. Liu, A. Cheng and H. Wang,
\newblock {\em {An hp-Galerkin method with fast solution for linear peridynamic models in one dimension}},
\newblock Comp. Math. Appl., 73(2017)1546-1565.

\bibitem{liu2020diagonal}
H. Liu, C. Sheng, L.-L. Wang, and H. Yuan,
\newblock {\em {On diagonal dominance of FEM stiffness matrix of fractional Laplacian and maximum principle preserving schemes for fractional Allen-Cahn equation}},
\newblock J. Sci. Comput., 86(2021)19.

\bibitem{Lu2022a}
J. Lu, Y. Nie, 
\newblock {\em {A reduced-order fast reproducing kernel collocation method for nonlocal models with inhomogeneous volume constraints}}, 
\newblock Comput. Math. Appl. 121(2022)52–61.



\bibitem{Mao2017}
 Z. Mao, J. Shen,
 \newblock {\em {Hermite spectral methods for fractional PDEs in unbounded domains}}, 
 \newblock SIAM J. Sci. Comput., 39(2017)A1928-A1950.

 \bibitem{Wang2014}
H. Wang, H. Tian,
\newblock {\em {A fast and faithful collocation method with efficient matrix assembly for a two-dimensional nonlocal diffusion model}},
\newblock Comput. Methods Appl. Mech. Engrg., 273(2014)19-36.

\bibitem{Sheng2020}
C. Sheng, J. Shen, T. Tang, L. L. Wang and H. Yuan,
\newblock {\em {Fast Fourier-like mapped Chebyshev spectral-Galerkin methods for PDEs with integral fractional Laplacian in unbounded domains}},
\newblock SIAM J. Numer. Anal., 58(2020)2435-2464.

\bibitem{Sheng2023}
C. Sheng, L. L. Wang, H. Chen and H. Li,
\newblock {\em {Fast implementation of FEM for integral fractional Laplacian on rectangular meshes}},
\newblock Commun. Comput. Phys., 35(2024).

 \bibitem{Tang2020}
 T. Tang, L.L. Wang, H. Yuan and T. Zhou, 
 \newblock {\em {Rational spectral methods for PDEs involving fractional Laplacian in unbounded domains}}, 
 \newblock SIAM J. Sci. Comput., 42 (2020)A585-A611.

\bibitem{Tian2013}
X. Tian and Q. Du,
\newblock {\em {Analysis and comparison of different approximations to nonlocal diffusion and linear peridynamic equations}},
\newblock SIAM J. Numer. Anal., 51(2013)3458-3482.

\bibitem{Tian2014}
X. Tian and Q. Du,
\newblock {\em {Asymptotically compatible schemes for robust discretization of nonlocal models and their local limits}},
\newblock SIAM J. Numer. Anal., 52(2014)1641-1665.

 \bibitem{Tian2016}
X. Tian, Q. Du and M. Gunzburger,
\newblock {\em {Asymptotically compatible schemes for the approximation of fractional Laplacian and related nonlocal diffusion problems on bounded domains}},
\newblock Adv. Comput. Math., 42(2016)1363–1380.

\bibitem{Trefethen2018ODE}
L. Trefethen, Á. Birkisson, and T. A.  Driscoll,
\newblock {\em {Exploring ODEs}}. 
\newblock Society for Industrial and Applied Mathematics, Philadelphia, PA, 2018.

\bibitem{Ye2023Monotone}
Q. Ye and X. Tian, 
\newblock {\em {Monotone meshfree methods for linear elliptic equations in non-divergence form via nonlocal relaxation}}, 
\newblock J. Sci. Comput. 96 (2023)33.


 
 \bibitem{Zheng2017}
C. Zheng, J. Hu, Q. Du, and J. Zhang,
\newblock {\em {Numerical solution of the nonlocal diffusion equation on the real line}},
\newblock SIAM J. Sci. Comput., 39(2017)A1951-A1968.

 \bibitem{Zhou2010}
K. Zhou and Q. Du,
\newblock {\em {Mathematical and numerical analysis of linear peridynamic models with nonlocal boundary conditions}},
\newblock SIAM J. Numer. Anal., 48(2010)1759-1780.

\end{thebibliography}
\end{document}